\documentclass[11pt]{article}
\usepackage[inner=3.4 cm,outer=3.4 cm,top=3.0 cm,bottom=2.0 cm]{geometry}

\usepackage{amsmath,amssymb,amsfonts,amsthm}
\usepackage{setspace}
\newenvironment{myabstract}{\par\noindent
{\bf Abstract . } \small }
{\par\vskip8pt minus3pt\rm}
\newcounter{item}[section]
\newcounter{kirshr}
\newcounter{kirsha}
\newcounter{kirshb}
\newenvironment{enumarab}{\setcounter{kirshb}{1}
\begin{list}{(\arabic{kirshb})}{\usecounter{kirshb}} }{\end{list}}
\newtheorem{theorem}{Theorem}[section]

\newtheorem{lemma}[theorem]{Lemma}
\newtheorem{corollary}[theorem]{Corollary}
\theoremstyle{definition}

\newtheorem{example}[theorem]{Example}
\newtheorem{definition}[theorem]{Definition}
\def\R{\mathbb{R}}
\def\Q{\mathbb{Q}}
\def\C{{\mathfrak{C}}}
\def\Fm{{\mathfrak{Fm}}}

\def\At{{\bf At}}
\def\Nr{{\mathfrak{Nr}}}

\def\Sg{{\mathfrak{Sg}}}
\def\Fm{{\mathfrak{Fm}}}
\def\A{{\mathfrak{A}}}
\def\B{{\mathfrak{B}}}
\def\C{{\mathfrak{C}}}
\def\D{{\mathfrak{D}}}
\def\M{{\mathfrak{M}}}
\def\N{{\mathfrak{N}}}
\def\Sn{{\mathfrak{Sn}}}
\def\CA{{\bf CA}}

\def\QEA{{\bf QEA}}

\def\Df{{\bf Df}}

\def\Lf{{\bf Lf}}

\def\PA{{\bf PA}}
\def\PEA{{\bf PEA}}

\def\K{{\bf K}}
\def\K{{\bf K}}

\def\RCA{{\bf RCA}}

\def\Rd{{\ Rd}}
\def\(R)RA{{\bf (R)RA}}
\def\RA{{\bf RA}}

\def\R{\mathbb{R}}
\def\Q{\mathbb{Q}}

\def\Sc{{\bf Sc}}

\def\Id{{\bf Id}}
\def\c #1{{\cal #1}}
 \def\CA{{\sf CA}}
\def\B{{\sf B}}
\def\G{{\sf G}}
\def\w{{\sf w}}
\def\y{{\sf y}}
\def\g{{\sf g}}
\def\b{{\sf b}}
\def\r{{\sf r}}
\def\K{{\sf K}}
\def\tp{{\sf tp}}

 \def\Cm{{\mathfrak{Cm}}}
\def\Nr{{\mathfrak{Nr}}}

\def\restr #1{{\restriction_{#1}}}
\def\cyl#1{{\sf c}_{#1}}
\def\diag#1#2{{\sf d}_{#1#2}}

\def\R{\sf R}
\def\Ra{{\mathfrak{Ra}}}
\def\Ca{{\mathfrak{Ca}}}
\def\set#1{\{#1\} }
\def\Ra{{\mathfrak{Ra}}}
\def\Nr{{\mathfrak{Nr}}}
\def\Tm{{\mathfrak{Tm}}}
\def\A{{\mathfrak{A}}}
\def\B{{\mathfrak{B}}}
\def\C{{\mathfrak{C}}}
\def\D{{\mathfrak{D}}}
\def\E{{\mathfrak{E}}}
\def\A{{\mathfrak{A}}}
\def\B{{\mathfrak{B}}}
\def\C{{\mathfrak{C}}}
\def\D{{\mathfrak{D}}}
\def\E{{\mathfrak{E}}}
\def\U{{\mathfrak{U}}}

\def\Bb{{\mathfrak{Bb}}}
\def\L{{\mathfrak{L}}}
\def\Rd{{\mathfrak{Rd}}}
\def\Bb{{\mathfrak{Bb}}}
\def\At{{\mathfrak{At}}}
\def\L{{\mathfrak{L}}}

\def\CA{{\bf CA}}
\def\RA{{\bf RA}}

\def\RCA{{\bf RCA}}

\def\G{{\bf G}}

\def\F{{\mathfrak{F}}}
\def\At{{\sf{At}}}
\def\N{\mathbb{N}}
\def\R{\mathfrak{R}}

\def\Cs{{\sf Cs}}

\def\cyl#1{{\sf c}_{#1}}
\def\diag#1#2{{\sf d}_{#1#2}}

\def\c #1{{\cal #1}}

\def\pa{$\forall$}
\def\pe{$\exists$}

\def\ef{Ehren\-feucht--Fra\"\i ss\'e}

\def\nodes{{\sf nodes}}

\def\restr #1{{\restriction_{#1}}}

\def\Ra{{\mathfrak{Ra}}}
\def\Nr{{\mathfrak{Nr}}}
\def\Z{{\cal Z}}
\def\CA{{\bf CA}}
\def\RCA{{\bf RCA}}
\def\c#1{{\mathcal #1}}

\def\set#1{ \{#1\}}

\def\Ca{{\mathfrak Ca}}
\def\b#1{{\bar{ #1}}}
\def\pe{$\exists$}
\def\pa{$\forall$}
\def\Cm{{\mathfrak Cm}}
\def\Sg{{\mathfrak Sg}}

\def\Rl{{\mathfrak Rl}}
\def\N{{\cal N}}

\def\ls { L\"owenheim--Skolem}
\def\At{{\sf At}}
\def\Uf{{\sf Uf}}

\def\rng{{\sf rng}}
\def\dom{{\sf dom}}
\def\Cm{{\sf Cm}}

\def\w{{\sf w}}
\def\g{{\sf g}}
\def\y{{\sf y}}
\def\r{{\sf r}}
\def\tp{{\sf tp}}
\def\cyl#1{{\sf c}_{#1}}

\def\diag#1#2{{\sf d}_{#1#2}}

\def\ws{winning strategy}
\def\ef{Ehren\-feucht--Fra\"\i ss\'e}

 \def\CA{{\sf CA}}
\def\Cs{{\sf Cs}}
\def\RCA{{\sf RCA}}

\def\RA{{\sf RA}}
\def\PA{{\sf PA}}

\def\PEA{\sf PEA}
\def\QEA{{\sf QEA}}

\def\y{{\sf y}}
\def\g{{\sf g}}
\def\r{{\sf r}}
\def\w{{\sf w}}
\def\Z{{\mathbb{Z}}}
\def\N{{\mathbb{N}}}

\def\U{{\mathfrak{U}}}

\def\Gp{{\sf Gp}}

\def\c{{\sf c}}
\def\s{{\sf s}}

\def\Id{{\sf Id}}
\def\Sc{{\sf Sc}}
\def\Df{{\sf Df}}
\def\Lf{{\sf Lf}}
\def\K{{\sf K}}
\def\nodes{{\sf nodes}}

\def\G{{\bold G}}
\def\Sc{{\sf Sc}}
\def\Df{{\sf Df}}
\def\PA{{\sf PA}}
\def\Id{{\sf Id}}
\def\QEA{{\sf QEA}}
\def\s{{\sf s}}

\def\CA{{\sf CA}}
\def\K{{\sf K}}

\def\RCA{{\sf RCA}}

\def\A{{\mathfrak{A}}}
\def\Cs{{\sf Cs}}

\def\cyl#1{{\sf c}_{#1}}

\def\diag#1#2{{\sf d}_{#1#2}}
\def\Mo{{\sf Mo}}

\def\la#1{\langle#1\rangle}
\def\G{{\mathfrak{G}}}
\def\Nr{{\sf Nr}}
\def\de{Dedekind-MacNeille}
\def\Cm{{\mathfrak{Cm}}}
\def\M{{\sf M}}
\def\T{{\sf T}}
\def\VT{{\sf VT}}
\def\CRCA{{\sf CRCA}}
\def\PEA{{\sf PEA}}
\def\Nrr{\mathfrak{Nr}}
\def\Tm{{\mathfrak{Tm}}}
\def\Mo{{\sf{M}}}
\title{An infinite stratum of representability; some cylindric algebras are more representable than others}
\author{Tarek Sayed Ahmed\\
Department of Mathematics, Faculty of Science,\\
Cairo University, Giza, Egypt.
 }
\date{}
\begin{document}
\maketitle

\begin{myabstract}   Let $2<n<m\leq \omega$. Let $\CA_n$ denote the class of cylindric algebras of dimension $n$ and $\RCA_n$ 
denote the class of representable $\CA_n$s. We say that $\A\in \RCA_n$ is representable up to $m$ if $\Cm\At\A$ has an $m$-square 
representation. An $m$ square represenation is  locally relativized represenation that is classical locally only on so called $m$-squares'. Roughly 
if we zoom in by a movable window to an $m$ square representation, there will become a point determinded and depending on $m$ 
where we mistake the $m$ square-representation for a genuine classical one. 
When we zoom out the non-representable part gets more exposed. 
For $2<n<m<l\leq \omega$, an $l$ square represenation is $m$-square; the converse however is not true. The variety $\RCA_n$ 
is a limiting case coinciding with $\CA_n$s having $\omega$-square representations. 
Let $\RCA_n^m$ be the class of algebras representable up to $m$. 
We show that $\RCA_n^{m+1}\subsetneq \bold \RCA_n^m$ for  $m\geq n+2$.
\footnote{Keywords: neat reducts, 
representations, degrees of representability.  
Mathematics subject classification: 03G15.}

\end{myabstract}

\section{Introduction}
Fix finite $n>2$. Let ${\sf CRCA}_n$ denote the class of completely representable $\CA_n$s and ${\sf LCA}_n={\bf El}\CRCA_n$ be the class of algebras satisfying the Lyndon conditions.
For a class $\sf K$ of Boolean algebras with operators, let $\sf K\cap \bf At$ denote the class of atomic algebras in $\sf K$.  By modifying the games coding the Lyndon conditions 
allowing \pa\ to reuse the pebble pairs on the board, we will show 
that ${\sf LCA}_n={\bf El}\CRCA_n={\bf El}\bold S_c\Nr_n\CA_{\omega}\cap \bf At$. 
Define an $A\in \CA_n$ to be {\it strongly representable} $\iff$ $\A$ is atomic and the complex algebra of its atom structure, equivalently its \de\ completion, in symbols $\Cm\At\A$ is in $\RCA_n$. 
This is a strong form of representability; of course $\A$ itself will be in $\RCA_n$, because $\A$ embeds into $\Cm\At\A$ and $\RCA_n$ is a variety, {\it a fortiori} closed under forming subalgebras.
We denote the class of strongly representable atomic algbras of dimension $n$ by ${\sf SRCA}_n$. Nevertheless,  there are atomic simple countable algebras that are representable, 
but not strongly representable. In fact, we shall see that there is a countable simple 
atomic algebra in $\RCA_n$ such that $\Cm\At\A\notin \bold S\Nr_n\CA_{n+3}(\supset \RCA_n)$. 
So in a way some algebras are more representable than others. In fact, the following inclusions are known to hold:  
$${\sf CRCA}_n\subsetneq {\sf LCA}_n\subsetneq {\sf SRCA}_n\subsetneq {\sf RCA}_n\cap \bf At.$$ 
In this paper we delve into a new notion, that of {\it degrees of representability}. Not all algebras are representable in the same way or strength. 
If $\C\subseteq \Nr_n\D$, with $\D\in \CA_m$ for some ordinal (possibly infinite) $m$,  we say that $\D$ is an $m$-dilation of $\C$ or simply a dilation if $m$ is clear from context. 
Using this jargon of 'dilating algebras' we  say that $\A\in {\sf RCA}_n$ is {\it strongly representable up to $m>n$}$\iff$ $\Cm\At\A$ admits an $m$- dilation equivalently $\Cm\At\A\in \bold S\Nr_n\CA_m$.-
This means that, though $\A$ itself is in $\RCA_n$,  the \de\ completion of $\A$ is not representable, but nevertheless it has some neat embedding property; it is `close' to bieng representable. 
The bigger the dimension of the dilation of the representable algebra, {\it the more representable} the algebra is, 
the closer it is to being strongly representable. The representability of an atomic algebra does not force its \de\ completion to be representable too, if it does then this algebra is strongly representable. 
A compelling question in this context is that if we let  $\langle \bold K_m: 2<n<m\leq \omega\rangle$ be the sequence whose $m$th entry $\bold K_m$ is the class of algebras 
that are strongly representable up to  $m$, 
it is obvious that this is a decreasing sequence, but is it stictly decreasing? In other words, are there $2<n<l<j\leq \omega$ such that $\bold K_l=\bold K_j?$ This question is far from being trivial,
and will be answered below. Through the unfolding of this paper, we will investigate and make precise the notion of an algebra being more representable than another.

\section{Preliminaries}
We follow the notation of \cite{1} which is in conformity with the notation in the monograph 
\cite{HMT2}. 
\begin{definition} 
Assume that $\alpha<\beta$ are ordinals and that 
$\B\in \CA_{\beta}$. Then the {\it $\alpha$--neat reduct} of $\B$, in symbols
$\mathfrak{Nr}_{\alpha}\B$, is the
algebra obtained from $\B$, by discarding
cylindrifiers and diagonal elements whose indices are in $\beta\setminus \alpha$, and restricting the universe to
the set $Nr_{\alpha}B=\{x\in \B: \{i\in \beta: {\sf c}_ix\neq x\}\subseteq \alpha\}.$
\end{definition}
It is straightforward to check that $\mathfrak{Nr}_{\alpha}\B\in \CA_{\alpha}$. 
Let $\alpha<\beta$ be ordinals. If $\A\in \CA_\alpha$ and $\A\subseteq \mathfrak{Nr}_\alpha\B$, with $\B\in \CA_\beta$, then we say that $\A$ {\it neatly embeds} in $\B$, and 
that $\B$ is a {\it $\beta$--dilation of $\A$}, or simply a {\it dilation} of $\A$ if $\beta$ is clear 
from context. For $\bold K\subseteq \CA_{\beta}$, 
we write $\Nr_{\alpha}\bold K$ for the class $\{\mathfrak{Nr}_{\alpha}\B: \B\in \bold K\}.$

Following \cite{HMT2}, ${\sf Cs}_n$ denotes the class of {\it cylindric set algebras of dimension $n$}, and ${\sf Gs}_n$ 
denotes the class of {\it generalized cylindric set algebra of dimension $n$}; $\C\in {\sf Gs}_n$, if $\C$ has top element
$V$ a disjoint union of cartesian squares,  that is $V=\bigcup_{i\in I}{}^nU_i$, $I$ is a non-empty indexing set, $U_i\neq \emptyset$  
and  $U_i\cap U_j=\emptyset$  for all $i\neq j$. The operations of $\C$ are defined like in cylindric set algebras of dimension $n$ 
relativized to $V$. It is known that $\bold I{\sf Gs}_n={\sf RCA}_n=\bold S\Nr_n\CA_{\omega}=\bigcap_{k\in \omega}\bold S\Nr_n\CA_{n+k}$. 
We often identify set algebras with their domain referring to an injection $f;\A\to \wp(V)$ ($\A\in \CA_n$) as a complete representation  of $\A$ (via $f$) 
where $V$  is a ${\sf Gs}_n$ unit.
\begin{definition} An algebra $\A\in {\sf CA}_n$ is {\it completely representable} $\iff$ there exists $\C\in {\sf Gs}_n$, and an isomorphism $f:\A\to \C$ such that for all $X\subseteq \A$, 
$f(\sum X)=\bigcup_{x\in X}f(x)$, whenever $\sum X$ exists in $\A$. In this case, we say that $\A$ is {\it completely representable via $f$.}
\end{definition}
It is known that $\A$ is completely representable via $f:\A\to \C$, where $\C\in {\sf Gs}_n$ has top element $V$ say 
$\iff$ $\A$ is atomic and $f$ is {\it atomic} in the sense that 
$f(\sum \At\A)=\bigcup_{x\in \At\A}f(x)=V$ \cite{HH}. We denote the class of completely representable $\CA_n$s by $\CRCA_n$.

To define certain deterministic games to be used in the sequel,  
we recall the notions of {\it atomic networks} and {\it atomic games} \cite{HHbook, HHbook2}. 
Let $i<n$. For $n$--ary sequences $\bar{x}$ and $\bar{y}$ $\iff \bar{y}(j)=\bar{x}(j)$ for all $j\neq i.$  

\begin{definition}\label{game} Fix finite $n>2$ and assume that $\A\in \CA_n$ is atomic.

(1) An {\it $n$--dimensional atomic network} on $\A$ is a map $N: {}^n\Delta\to  At\A$, where
$\Delta$ is a non--empty set of {\it nodes}, denoted by $\nodes(N)$, satisfying the following consistency conditions for all $i<j<n$: 
\begin{itemize}
\item If $\bar{x}\in {}^n\nodes(N)$  then $N(\bar{x})\leq {\sf d}_{ij}\iff x_i=x_j$,
\item If $\bar{x}, \bar{y}\in {}^n\nodes(N)$, $i<n$ and $\bar{x}\equiv_i \bar{y}$, then  $N(\bar{x})\leq {\sf c}_iN(\bar{y})$.
%\item (Symmetry): if $\bar{x}\in {}^n\nodes(N)$, then  $\s_{[i, j]}N(\bar{x})=N(\bar{x}\circ [i, j]).$
\end{itemize}
For $n$--dimensional atomic networks $M$ and $N$,  we write $M\equiv_i N\iff M(\bar{y})=N(\bar{y})$ for all $\bar{y}\in {}^{n}(n\sim \{i\})$.

(2)    Assume that $m, k\leq \omega$. 
The {\it atomic game $G^m_k(\At\A)$, or simply $G^m_k$}, is the game played on atomic networks
of $\A$ using $m$ nodes and having $k$ rounds \cite[Definition 3.3.2]{HHbook2}, where
\pa\ is offered only one move, namely, {\it a cylindrifier move}: 
Suppose that we are at round $t>0$. Then \pa\ picks a previously played network $N_t$ $(\nodes(N_t)\subseteq m$), 
$i<n,$ $a\in \At\A$, $x\in {}^n\nodes(N_t)$, such that $N_t(\bar{x})\leq {\sf c}_ia$. For her response, \pe\ has to deliver a network $M$
such that $\nodes(M)\subseteq m$,  $M\equiv _i N$, and there is $\bar{y}\in {}^n\nodes(M)$
that satisfies $\bar{y}\equiv _i \bar{x}$ and $M(\bar{y})=a$.  
We write $G_k(\At\A)$, or simply $G_k$, for $G_k^m(\At\A)$ if $m\geq \omega$.
\end{definition}

\subsection{Clique guarded semantics}
Fix $2<n<\omega$,  
We study three approaches to approximating the class $\sf RCA_n$  by (a) basis,  (b) existence of 
dilations and finally (c) (locally well--behaved) relativized representations, in analogy to the relation algebra case dealt with in 
\cite[Chapter 13]{HHbook}. Examples include $m$--flat and $m$--square representations, where $2<n<m<\omega$. It will always be the case, unless otherwise explicitly indicated,
that $1<n<m<\omega$;  $n$ denotes the dimension. 
But first we recall certain relativized set algebras. A set $V$ ($\subseteq {}^nU$)  is {\it diagonizable} if  
$s\in V\implies s\circ [i|j]\in V$.
We say that $V\subseteq {}^nU$ is {\it locally square} if whenever $s\in V$ and
$\tau: n\to n$, then $s\circ \tau\in V$. 
Let ${\sf D}_n$ (${\sf G}_n$) be the class of set 
algebras whose top elements are diagonizable (locally square) and operations are defined like cylindric set algebra 
of dimension $n$ relativized to the top element $V$ . 
We identify notationally a 
set algebra with its universe.  Let $\Mo$ be a {\it relativized representation} of $\A\in \CA_n$, that is, there exists an injective
homomorphism $f:\A\to \wp(V)$ where $V\subseteq {}^n\Mo$ and $\bigcup_{s\in V} \rng(s)=\Mo$. For $s\in V$ and $a\in \A$,
we may write $a(s)$ for $s\in f(a)$. This notation does not refer to $f$, but whenever used 
then  either $f$ will be clear from context, or immaterial in the context. We may also write $1^{\Mo}$ for $V$.  
Let  $\L(\A)^m$ be the first order signature using $m$ variables
and one $n$--ary relation symbol for each element of $\A$.  Allowing infinitary conjunctions, we denote the resulting signature taken in $L_{\infty, \omega}$
by $\L(\A)_{\infty, \omega}^m$.

{\it An $n$--clique}, or simply a clique,  is a set $C\subseteq \Mo$ such
$(a_0,\ldots, a_{n-1})\in V=1^{\Mo}$
for all distinct $a_0, \ldots, a_{n-1}\in C.$
Let
$${\sf C}^m(\Mo)=\{s\in {}^m\Mo :\rng(s) \text { is an $n$ clique}\}.$$
Then ${\sf C}^m(\Mo)$ is called the {\it $n$--Gaifman hypergraph}, or simply Gaifman hypergraph  of $\Mo$, with the $n$--hyperedge relation $1^{\Mo}$.
The {\it $n$-clique--guarded semantics}, or simply clique--guarded semantics,  $\models_c$, are defined inductively. 
Let $f$ be as above. For an atomic $n$--ary formula $a\in \A$, $i\in{}^nm$, 
and $s\in {}^m\Mo$, $\Mo, s\models_c a(x_{i_0},\ldots x_{i_{n-1}})\iff\ (s_{i_0}, \ldots s_{i_{n-1}})\in f(a).$ 
For equality, given $i<j<m$, $\Mo, s\models_c x_i=x_j\iff s_i=s_j.$
Boolean connectives, and infinitary disjunctions,
are defined as expected.  Semantics for existential quantifiers
(cylindrifiers) are defined inductively for $\phi\in \L(A)^m_{\infty, \omega}$ as follows:
For $i<m$ and $s\in {}^m\Mo$, $\Mo, s\models_c \exists x_i\phi \iff$ there is a $t\in {\sf C}^m(\Mo)$, $t\equiv_i s$ such that 
$\Mo, t\models_c \phi$. 
\begin{definition}\label{cl}
Let $\A\in \CA_n$, $\Mo$ a relativized representation of $\A$ and $\L(\A)^m$  be as above. 
\begin{enumarab}
\item Then $\Mo$ is said to be {\it $m$--square},
if witnesses for cylindrifiers can be found on $n$--cliques. More precisely,
for all  $\bar{s}\in {\sf C}^m(\Mo), a\in \A$, $i<n$,
and for any injective map  $l:n\to m$, if $\Mo\models {\sf c}_ia(s_{l(0)}\ldots, s_{l(n-1)})$,
then there exists $\bar{t}\in {\sf C}^m(\Mo)$ with $\bar{t}\equiv _i \bar{s}$,
and $\Mo\models a(t_{l(0)},\ldots, t_{l(n-1)})$.

\item $\Mo$ is said to be {\it (infinitary) $m$--flat} if  it is $m$--square and
for all $\phi\in (\L(A)_{\infty, \omega}^m) \L(\A)^m$, 
for all $\bar{s}\in {\sf C}^m(\Mo)$, for all distinct $i,j<m$,
we have
$\Mo\models_c [\exists x_i\exists x_j\phi\longleftrightarrow \exists x_j\exists x_i\phi](\bar{s}).$
\end{enumarab}
\end{definition}
We also need the notion of $m$--dimensional hyperbasis. This hyperbasis is made up of $m$--dimensional hypernetworks.
An $m$--dimensional hypernetwork on the atomic algebra $\A$ is an $n$--dimensional  network $N$, with $\nodes(N)\subseteq m$, endowed with a set of labels $\Lambda$ for
hyperedges of length $\leq m$,
not equal to $n$ (the dimension), such that $\Lambda\cap \At\A=\emptyset$. We call a label in $\Lambda$ a non-atomic label.
Like in networks, $n$--hyperedges are labelled by atoms. In addition to the consistency properties for networks,
an $m$--dimensional hypernetwork should satisfy the following additional consistency rule involving non--atomic labels:
If $\bar{x}, \bar{y}\in {}^{\leq m}m$, $|\bar{x}|=|\bar{y}|\neq n$ and $\exists \bar{z}$, such that $\forall i<|\bar{x}|$,
$N(x_i,y_i,\bar{z})\leq {\sf d}_{01}$,
then $N(\bar{x})=N(\bar{y})\in \Lambda$. 
%(We shall deal with hypernetworks in item (1) of theorem \ref{main}).

\begin{definition} Let $2<n<m<\omega$ and $\A\in \CA_n$ be atomic.

(1) An $m$--dimensional basis $B$ for $\A$ consists of a set of $n$--dimensional networks whose nodes $\subseteq m$, satisfying 
the following properties: 

\begin{itemize}

\item For all  $a\in \At\A$, there is an $N\in B$ such that $N(0,1,\ldots, n-1)=a,$

\item The {\it cylindrifier property}: For all $N\in B$, all $i<n$,  all $\bar{x}\in {}^n\nodes(N)(\subseteq {}^nm)$, all $a\in\At\A$, such that
$N(\bar{x})\leq {\sf c}_ia$,  there exists $M\in B$, $M\equiv_i N$, $\bar{y}\in {}^n\nodes(M)$ such 
that $\bar{y}\equiv_i\bar{x}$ and $M(\bar{y})=a.$  We can always assume that $\bar{y}_i$ is a new node else one takes $M=N$. 

\end{itemize}

(2)  An $m$--dimensional  hyperbasis $H$ consists of $m$--dimensional hypernetworks, satisfying the above two conditions reformulated 
the obvious way for hypernetworks, in addition,  $H$ has an amalgamation property for overlapping hypernertworks; 
this property corresponds to commutativity of cylindrifiers: 

For all $M,N\in H$ and $x,y<m$, with $M\equiv_{xy}N$, there is $L\in H$ such that
$M\equiv_xL\equiv_yN$. Here $M\equiv_SN$, means 
that $M$ and $N$ agree off of $S$ \cite[Definition 12.11]{HHbook}.
\end{definition}
%One can define $m$--smooth representations as in \cite[Definition 13.12] {HHbook}.
%But like infinitary $m$--flat representations,  they really do not add much to $m$--flatness as far as (ordinary as opposed to complete) 
%representations are concerned.
%But they do in the case of {\it complete} $m$--relativized representations, witness item (3) in the coming lemma.
%An algebra having a complete $m$--flat representation, has an $m$--infinitary flat representation, 
%but not necessarily a complete one. 
%One more  technical definition to be used in the coming Lemma interconnecting existence of $m$- relativized representations, $m$ dimensional basis and $m$-dilations, where $2<n<m<\omega$.

\begin{definition}\label{sub} Let $m$ be a finite ordinal $>0$. An $\sf s$ word is a finite string of substitutions $({\sf s}_i^j)$ $(i, j<m)$,
a $\sf c$ word is a finite string of cylindrifications $({\sf c}_i), i<m$;
an $\sf sc$ word $w$, is a finite string of both, namely, of substitutions and cylindrifications.
An $\sf sc$ word
induces a partial map $\hat{w}:m\to m$:
\begin{itemize}

\item $\hat{\epsilon}=Id,$

\item $\widehat{w_j^i}=\hat{w}\circ [i|j],$

\item $\widehat{w{\sf c}_i}= \hat{w}\upharpoonright(m\smallsetminus \{i\}).$

\end{itemize}
If $\bar a\in {}^{<m-1}m$, we write ${\sf s}_{\bar a}$, or
${\sf s}_{a_0\ldots a_{k-1}}$, where $k=|\bar a|$,
for an  arbitrary chosen $\sf sc$ word $w$
such that $\hat{w}=\bar a.$
Such a $w$  exists by \cite[Definition~5.23 ~Lemma 13.29]{HHbook}.
\end{definition}
The proof of the  following lemma can be distilled
from its $\sf RA$ analogue \cite[Theorem 13.20]{HHbook},  by reformulating deep concepts
originally introduced by Hirsch and Hodkinson for $\sf RA$s in the $\CA$ context, involving the notions of 
hypernetworks and hyperbasis. This can (and will) be done.
In the coming proof, we highlight
the main ideas needed to perform such a transfer from $\sf RA$s to $\CA$s
\cite[Definitions 12.1, 12.9, 12.10, 12.25, Propositions 12.25, 12.27]{HHbook}. 
In all cases, the $m$--dimensional dilation stipulated in the statement of the theorem, will have
top element ${\sf C}^m(\Mo)$, where $\Mo$ is the $m$--relativized representation of the given algebra, and the operations of the dilation
are induced by the $n$-clique--guarded semantics. For a class $\sf K$ of $\sf BAO$s, $\sf K\cap \bf At$ denotes the class of atomic algebras in $\sf K$.
\begin{lemma}\label{flat}\cite{atom} \cite[Theorems 13.45, 13.36]{HHbook}.
Assume that $2<n<m<\omega$ and let $\A\in \CA_n$. Then $\A\in \bold S\Nr_n\CA_m\iff \A$ has an  infinitary $m$--flat representation
$\iff \A$ has an $m$--flat representation. Furthermore, 
if $\A$ is atomic, then $\A$ has a complete infinitary $m$--flat representation $\iff$ $\A\in \bold S_c\Nr_n(\CA_m\cap \bf At)$.
We  can replace infinitary $m$-flat and $\CA_m$ by $m$-square and ${\sf D}_m$, respectively.
\end{lemma}

\begin{proof} We give a sketchy sample. More details can be found in \cite{atom}.
We start from representations to  dilations.
Let $\Mo$ be an $m$--flat representation of $\A$. 
For $\phi\in \L(\A)^m$, 
let $\phi^{\Mo}=\{\bar{a}\in {\sf C}^m(\Mo):\Mo\models_c \phi(\bar{a})\}$, where ${\sf C}^m(\Mo)$ is the $n$--Gaifman hypergraph.
Let $\D$ be the algebra with universe $\{\phi^{M}: \phi\in \L(\A)^m\}$ and with  cylindric
operations induced by the $n$-clique--guarded (flat) semantics. 
For $r\in \A$, and $\bar{x}\in {\sf C}^m(\Mo)$, we identify $r$ with the formula it defines in $\L(\A)^m$, and 
we write $r(\bar{x})^{\Mo}\iff \Mo, \bar{x}\models_c r$.
Then $\D$ is a set algebra 
with domain $\wp({\sf C}^m(\Mo))$ and with unit $1^{\D}={\sf C}^m(\Mo)$.
Since $\Mo$ is $m$--flat, then cylindrifiers in $\D$ commute, and so $\D\in \CA_m$.
Now define $\theta:\A\to \D$, via $r\mapsto r(\bar{x})^{\Mo}$. Then exactly like in the proof of \cite[Theorem 13.20]{HHbook},
$\theta$ is an injective neat embedding, that is, $\theta(\A)\subseteq \mathfrak{Nr}_n\D$.
 The relativized model $\Mo$ itself might not be  infinitary $m$--flat, but one can build an infinitary $m$--flat representation of $\A$, whose base $\Mo$ is an $\omega$--saturated model
of the consistent first order theory, stipulating the existence of an $m$--flat representation, cf. \cite[Proposition 13.17, Theorem 13.46 items (6) and (7)]{HHbook}.

The inverse implication from dilations to representations harder. One constructs from the given 
$m$--dilation, an $m$--dimensional 
hyperbasis (that can be defined similarly to the $\RA$ case, cf. \cite[Definition 12.11]{HHbook}) 
from which
the required $m$-relativized representation is built.  
This can be done in a step--by step manner treating the hyperbasis 
as a `saturated set of mosaics', cf. \cite[Proposition 13.37]{HHbook}..
We show how an $m$--dimensional hyperbasis for the  canonical extension of $\A\in \CA_n$
is obtained from an $m$--dilation of $\A$ \cite[Definition 13.22, lemmata 13.33-34-35, Proposition 36]{HHbook}.
Suppose that $\A\subseteq \Nrr_n\D$ for some $\D\in \CA_m$.
Then $\A^+\subseteq_c \Nrr_m\D^+$, and $\D^+$ is atomic. We show that $\D^+$ has an $m$--dimensional hyperbasis.
First, it is not hard to see that for every $n\leq l\leq m$, $\Nrr_l\D^+$ is atomic.
The set of non--atomic labels $\Lambda$ is the set $\bigcup_{k<m-1}\At\Nrr_k\D^+$.
For each atom $a$ of $\D^+$, define a labelled  hypergraph $N_a$ as follows.
Let $\bar{b}\in {}^{\leq m}m$. Then if $|\bar{b}|=n$,  so that $\bar{b}$  has to get a label that is an atom of $\D^+$, one sets  $N_a(\bar{b})$ to be 
the unique $r\in \At\D^+$ such that $a\leq {\sf s}_{\bar{b}}r$; notation here
is given in definition \ref{sub}.
If $n\neq |\bar{b}| <m-1$, $N_a(\bar{b})$ is the unique atom $r\in \Nrr_{|b|}\D^+$ such that $a\leq {\sf s}_{\bar{b}}r.$ Since
$\Nrr_{|b|}\D^+$ is atomic, this is well defined. Note that this label may be a non--atomic one; it 
might not be an atom of $\D^+$. But by definition it is a permitted label.
Now fix $\lambda\in \Lambda$. The rest of the labelling is defined by $N_a(\bar{b})=\lambda$.
Then $N_a$ as an $m$--dimensional
hypernetwork, for each 
such chosen $a$,  and $\{N_a: a\in \At\D^+\}$ is the required $m$--dimensional hyperbasis.
The rest of the proof consists of a fairly straightforward adaptation of the proof \cite[Proposition 13.37]{HHbook},
replacing edges by $n$--hyperedges.

For results on {\it complete} $m$--flat representations, one works in $L_{\infty, \omega}^m$ instead of first order logic. 
With $\D$ formed like above from (the complete $m$--flat representation) $\Mo$, using $\L(\A)_{\infty,\omega}^m$ instead of $L_n$, 
let $\phi^{\Mo}$ be a non--zero element in $\D$.
Choose $\bar{a}\in \phi^{\Mo}$, and let $\tau=\bigwedge \{\psi\in \L(\A)_{\infty,\omega}^m: \Mo\models_c \psi(\bar{a})\}.$
Then $\tau\in \L(\A)_{\infty,\omega}^m$, and $\tau^{\Mo}$ is an atom below $\phi^{\Mo}$. 
The rest is entirely analogous, cf. \cite[p.411]{HHbook}. 

\end{proof}

The following lemma  
is proved in \cite[Lemma 5.8]{mlq}
\begin{lemma}\label{n}Let $2<n<m$. 
\item If $\A\in \CA_n$ is finite and \pa\ has a \ws\ in $G^m_{\omega}(\At\A),$ then $\A$  does not have an $m$--square representation. 
%In particular, if $\A$ is finite, then $\A$ does not have an $m$--square representation. 
\end{lemma}
%In the next Theorem we use a rainbow construction. We follow \cite[\S 4.3]{HH} for terminology and notation.
In our next proof we use a {\it rainbow constructions}; in this we follow \cite{HH, HHbook2}.
Fix $2<n<\omega$. Given relational structures 
$\sf G$ (the greens) and $\sf R$ (the reds) the rainbow 
atom structure of a $\CA_n$  consists of equivalence classes of surjective maps $a:n\to \Delta$, where $\Delta$ is a coloured graph.
A {\it coloured graph} is a complete graph labelled by the rainbow colours, the greens $\g\in \sf G$,  reds $\r\in \sf R$, and whites; and some $n-1$ tuples are labelled by `shades of yellow'.
In coloured graphs certain triangles are not allowed for example all green triangles are forbidden.
A red triple $(\r_{ij}, \r_{j'k'}, \r_{i^*k^*})$ $i,j, j', k', i^*, k^*\in \sf R$ is not allowed,  unless $i=i^*,\; j=j'\mbox{ and }k'=k^*$, 
in which case we say that the  red indices match, cf.\cite[4.3.3]{HH}.
The equivalence relation relates two such maps $\iff$  they essentially define the same graph \cite[4.3.4]{HH}. We let $[a]$ denote the equivalence class containing $a$. 
For $2<n<\omega$, we use the graph version of the usual atomic $\omega$--rounded game $G_{\omega}^m(\alpha)$ with $m$ nodes, played on atomic networks
of the $\CA_n$ atom  structure $\alpha$. 
The game  $\bold G^m(\beta)$ where $\beta$ is a $\CA_n$ 
atom structure is like $G_{\omega}^m(\At\A)$ except that \pa\ has the option to reuse the $m$ nodes in play. We use the `graph versions' of these games,
cf. \cite[4.3.3]{HH}.  The (complex) rainbow algebra based on $\sf G$ and $\sf R$ is denoted by $\A_{\sf G, R}$.
The dimension $n$ will always be clear from context.

\section{Degrees of representability}

We let  $\bold S_c$ denotes the operation of forming complete sublgebras and $\bold S_d$ denotes the operation of forming dense subalgebras. We let 
$\bold I$ denote the operation of forming 
isomorphic images.  
For any class of $\sf BAO$s
$\bold I{\sf K}\subseteq \bold S_d{\sf K}\subseteq \bold S_c\sf K$. (It is not hard to show that for Boolean algebras the inclusion are proper).
\begin{definition}\label{complex} Let $2<n\leq l\leq m\leq \omega$. Let $\bold O\in \{\bold S, \bold S_d, \bold S_c, \bold I\}$.
\begin{enumarab}
\item An algebra $\A\in \CA_n$ has the {\it $\bold O$ neat embedding property up to $m$} if $\A\in \bold O\Nr_n\CA_m$. If $m=\omega$ and $\bold O=\bold S$, 
we say simply that $\A$ has {\it the 
neat embedding property}. (Observe that the last condition is equivalent to that $\A\in \RCA_n$). 

\item An atomic algebra $\A\in \sf CA_n$ has the {\it complex $\bold O$ neat embedding property up to $m$}, if $\Cm\At\A\in \bold O\Nr_n\CA_m.$ 
The word `complex' here refers to the {\it involvement of the complex algebra} 
in the definition.

\item An atomic algebra $\A\in {\sf RCA}_n$ is {\it strongly representable up to $l$ and $m$} if $\A\in \Nr_n\CA_l$ 
and $\Cm\At\A\in \bold S\Nr_n\CA_m$. If $l=n$ and $m=\omega$, we say that $\A$ is 
{\it strongly representable}.
\end{enumarab}
\end{definition}  
\begin{theorem}\label{can}
Then is an atomic simple countable $\A\in \CA_n$ (i.e has the neat embedding property) 
but not the complex $\bold S$ neat embedding propery up to $m$ for any $m\geq n+3$.  
\end{theorem}
\begin{proof} We show that there is a countable atomic $\A\in \RCA_n$ 
such that $\Cm\At\A$ does not have an $n+3$--square representation.
This is proved in \cite{mlq} in the context of omituing types. Here we give a direct shorter more streamlined proof. The idea however is essentially the same. 
Take the finite rainbow cylindric algebra $R(\Gamma)$
as defined in \cite[Definition 3.6.9]{HHbook2},
where $\Gamma$ (the reds) is taken to be the complete irreflexive graph $m$, and the greens
are  $\{\g_i:1\leq i<n-1\}
\cup \{\g_0^{i}: 1\leq i\leq n+1\}$ so that $\sf G$ is the complete irreflexive graph $n+1$.

Call this finite rainbow $n$--dimensional cylindric algebra, based on ${\sf G}=n+1$ and ${\sf R}=n$,
$\CA_{n+1, n}$  and denote its finite atom structure by $\bf At_f$.
One  then replaces each  red colour
used in constructing  $\CA_{n+1, n}$ by infinitely many with superscripts from $\omega$, 
getting a weakly representable atom structure $\bf At$, that is,
the term algebra $\Tm\bf At$ is representable.
The resulting atom structure (with $\omega$--many reds),  call it $\bf At$, 
is the rainbow atom structure that is like the atom structure of the (atomic set) algebra denoted by $\A$ in \cite[Definition 4.1]{Hodkinson} except that we have $n+1$ greens
and not infinitely many as is the case in \cite{Hodkinson}.
Everything else is the same. In particular, the rainbow signature \cite[Definition 3.6.9]{HHbook2} now consists of $\g_i: 1\leq i<n-1$, $\g_0^i: 1\leq i\leq n+1$,
$\w_i: i<n-1$,  $\r_{kl}^t: k<l< n$, $t\in \omega$,
binary relations, and $n-1$ ary relations $\y_S$, $S\subseteq n+1$.
There is a shade of red $\rho$; the latter is a binary relation that is {\it outside the rainbow signature}.
But $\rho$ is used as a  label for  coloured graphs built during a `rainbow game', and in fact, \pe\ can win the rainbow $\omega$--rounded game
and she builds an $n$--homogeneous (coloured graph) model $M$ as indicated in the above outline by using $\rho$ when
she is forced a red \cite[Proposition 2.6, Lemma 2.7]{Hodkinson}.
Then, it can be shown exactly as in \cite{Hodkinson}, that $\Tm\At$ is representable as a set algebra with unit $^nM$.
We give more details. In the present context, after the splitting `the finitely many red colours' replacing each such red colour $\r_{kl}$, $k<l<n$  by $\omega$ many 
$\r_{kl}^i$, $i\in \omega$, the rainbow signature for the resulting rainbow theory as defined in \cite[Definition 3.6.9]{HHbook} call this theory $T_{ra}$,
 consists of $\g_i: 1\leq i<n-1$, $\g_0^i: 1\leq i\leq n+1$,
$\w_i: i<n-1$,  $\r_{kl}^t: k<l< n$, $t\in \omega$,
binary relations, and $n-1$ ary relations $\y_S$, $S\subseteq_{\omega} n+k-2$ or $S=n+1$. 
The set algebra ${\mathfrak{Bb}}(\A_{n+1, n}, \r, \omega)$ of dimension $n$ has
base an $n$--homogeneous  model $\Mo$ of another theory $T$ whose signature expands that of 
$T_{ra}$ by an additional binary relation (a shade of red) $\rho$.  
In this new signature $T$ is obtained from $T_{ra}$ by 
some axioms  (consistency conditions) extending $T_{ra}$. Such axioms (consistency conditions) 
specify consistent triples involving $\rho$. We call the models of $T$ {\it extended} coloured graphs. 
In particular, $\Mo$ is an extended coloured graph.
To build $\Mo$, the class of coloured graphs is considered in
the signature $L\cup \{\rho\}$ like in uual rainbow constructions as given above with the two additional forbidden triples
$(\r, \rho, \rho)$ and $(\r, \r^*, \rho)$, where $\r, \r^*$ are any reds. 
%Let $\GG$ be the class of all models of this {\it extended rainbow first order theory}.
%The extra shade of red $\rho$  will be used as a label.
This model $\Mo$ is constructed as a countable limit of finite models of $T$ 
using a game played between \pe\ and \pa.   Here, unlike the extended $L_{\omega_1, \omega}$ theory 
dealt with in \cite{Hodkinson},  $T$ is a {\it first order one}
because the number of greens used are finite.
In the rainbow game \cite{HH, HHbook} 
\pa\ challenges \pe\  with  {\it cones} having  green {\it tints $(\g_0^i)$}, 
and \pe\ wins if she can respond to such moves. This is the only way that \pa\ can force a win.  \pe\ 
has to respond by labelling {\it appexes} of two succesive cones, having the {\it same base} played by \pa.
By the rules of the game, she has to use a red label. She resorts to  $\rho$ whenever
she is forced a red while using the rainbow reds will lead to an inconsistent triangle of reds;  \cite[Proposition 2.6, Lemma 2.7]{Hodkinson}.

We next embed $\CA_{n+1, n}$ into  the complex algebra $\Cm\bf At$, the \de\ completion of $\Tm\bf At$.
Let ${\sf CRG}_f$ denote  the class of coloured graphs on 
$\bf At_f$ and $\sf CRG$ be the class of coloured graph on $\bf At$. We 
can assume that  ${\sf CRG}_f\subseteq \sf CRG$.
Write $M_a$ for the atom that is the (equivalence class of the) surjection $a:n\to M$, $M\in \sf CRG$.
Here we identify $a$ with $[a]$; no harm will ensue.
We define the (equivalence) relation $\sim$ on $\At$ by
$M_b\sim N_a$, $(M, N\in {\sf CRG})$ $\iff$ they are everywhere identical except possibly at red edges:
$$M_a(a(i), a(j))=\r^l\iff N_b(b(i), b(j))=\r^k,  \text { for some $l,k$}\in \omega.$$
We say that $M_a$ is a {\it copy of $N_b$} if $M_a\sim N_b$. 
Now we define a map $\Theta: \CA_{n+1, n}=\Cm{\bf At_f}$ to $\Cm\At$,
by  specifing  first its values on ${\sf At}_f$,
via $M_a\mapsto \sum_jM_a^{(j)}$; where $M_a^{(j)}$ is a copy of $M_a$; each atom maps to the suprema of its 
copies.  (If $M_a$ has no red edges,  then by $\sum_jM_a^{(j)}$,  we understand $M_a$).
This map is extended to $\CA_{n+1, n}$ the obvious way. The map
$\Theta$ is well--defined, because $\Cm\At$ is complete. 
It is not hard to show that the map $\Theta$ 
is an injective homomorphim. 
We check preservation of all the $\QEA_n$ operations.  
The Boolean join is obvious.
\begin{itemize}
\item For complementation: It suffices to check preservation of  complementation `at atoms' of ${\bf At}_f$. 
So let $M_a\in {\bf At}_f$ with $a:n\to M$, $M\in \sf CGR_f\subseteq \sf CGR$. Then: 

$$\Theta(\sim M_a)=\Theta(\bigcup_{[b]\neq [a]} M_b)
=\bigcup_{[b]\neq [a]} \Theta(M_b)
=\bigcup_{[b]\neq [a]}\sum_j M_b^{(j)}$$
$$=\bigcup_{[b]\neq [a]}\sim \sum_j[\sim (M_a)^{(j)}]
=\bigcup_{[b]\neq [a]}\sim \sum_j[(\sim M_b)^j]
=\bigcup_{[b]\neq [a]}\bigwedge_j M_b^{(j)}$$
$$=\bigwedge_j\bigcup_{[b]\neq [a]}M_b^{(j)}
=\bigwedge_j(\sim M_a)^{j}
=\sim (\sum M_a^j)
=\sim \Theta(a)$$

\item Diagonal elements. Let $l<k<n$. Then:
\begin{align*}
M_x\leq \Theta({\sf d}_{lk}^{\Cm{\bf At}_f})&\iff\ M_x\leq \sum_j\bigcup_{a_l=a_k}M_a^{(j)}\\
&\iff M_x\leq \bigcup_{a_l=a_k}\sum_j M_a^{(j)}\\
&\iff  M_x=M_a^{(j)}  \text { for some $a: n\to M$ such that $a(l)=a(k)$}\\
&\iff M_x\in {\sf d}_{lk}^{\Cm\bf At}.
\end{align*}

\item Cylindrifiers. Let $i<n$. By additivity of cylindrifiers, we restrict our attention to atoms 
$M_a\in {\bf At}_f$ with $a:n\to M$, and $M\in {\sf CRG}_f\subseteq \sf CRG$. Then: 

$$\Theta({\sf c}_i^{\Cm{\bf At}_f}M_a)=f (\bigcup_{[c]\equiv_i[a]} M_c)
=\bigcup_{[c]\equiv_i [a]}\Theta(M_c)$$
$$=\bigcup_{[c]\equiv_i [a]}\sum_j M_c^{(j)}=\sum_j \bigcup_{[c]\equiv_i [a]}M_c^{(j)}
=\sum _j{\sf c}_i^{\Cm\bf At}M_a^{(j)}$$
$$={\sf c}_i^{\Cm\bf At}(\sum_j M_a^{(j)})
={\sf c}_i^{\Cm\bf At}\Theta(M_a).$$

\end{itemize}

It is straightforward to show that 
\pa\ has \ws\ first in the  \ef\ forth  private 
game played between \pe\ and \pa\ on the complete
irreflexive graphs $n+1$ and $n$ in 
$n+1$ rounds
${\sf EF}_{n+1}^{n+1}(n+1, n)$ \cite [Definition 16.2]{HHbook2}
since $n+1$ is `longer' than $n$. 
Here $r$ is the number of rounds and $p$ is the number of pairs of pebbles
on board. Using (any) $p>n$ many pairs of pebbles avalable on the board \pa\ can win this game in $n+1$ many rounds.
In each round $0,1\ldots n$, \pe\ places a new pebble  on  a new element of $n+1$.
The edge relation in $n$ is irreflexive so to avoid losing
\pe\ must respond by placing the other  pebble of the pair on an unused element of $n$.
After $n$ rounds there will be no such element, so she loses in the next round.
 \pa\  lifts his \ws\ from the private \ef\ forth game ${\sf EF}_{n+1}^{n+1}(n+1, n)$ to the graph game on ${\bf At}_f=\At(\A_{n+1,n})$ 
\cite[pp. 841]{HH} forcing a
win using $n+3$ nodes. 
He bombards \pe\ with cones
having  common
base and distinct green  tints until \pe\ is forced to play an inconsistent red triangle (where indicies of reds do not match).
Thus \pa\ has  a 
\ws\ for \pe\ in $\bold G^{n+3}\At(\CA_{n+1, n})$
using the usual rainbow strategy by bombarding \pe\ with cones having the same base and distinct green tints.
He needs $n+3$ nodes to implement his \ws. In fact he need $n+3$ nodes to force a win in the weaker game $G^{n+3}_{\omega}$ without the need to
resue the nodes in play.
Then by Lemma \ref{n}, this implies that  $\CA_{n+1,n}$ does not have an $n+3$--square representation. Since $\CA_{n+1,n}$ embeds into $\Cm\bf At$, 
hence $\Cm\bf At$  does not have an $n+3$--square representation, too.
\end{proof}

The following definition to be used in the sequel is taken from \cite{ANT}:
\begin{definition}\label{strongblur}\cite[Definition 3.1]{ANT}
Let $\R$ be a relation algebra, with non--identity atoms $I$ and $2<n<\omega$. Assume that  
$J\subseteq \wp(I)$ and $E\subseteq {}^3\omega$.
\begin{enumerate}
\item We say that $(J, E)$  is an {\it $n$--blur} for $\R$, if $J$ is a {\it complex $n$--blur} defined as follows:   
\begin{enumarab}
\item Each element of $J$ is non--empty,
\item $\bigcup J=I,$
\item $(\forall P\in I)(\forall W\in J)(I\subseteq P;W),$
\item $(\forall V_1,\ldots V_n, W_2,\ldots W_n\in J)(\exists T\in J)(\forall 2\leq i\leq n)
{\sf safe}(V_i,W_i,T)$, that is there is for $v\in V_i$, $w\in W_i$ and $t\in T$,
we have
$v;w\leq t,$ 
\item $(\forall P_2,\ldots P_n, Q_2,\ldots Q_n\in I)(\forall W\in J)W\cap P_2;Q_n\cap \ldots P_n;Q_n\neq \emptyset$.
\end{enumarab}
and the tenary relation $E$ is an {\it index blur} defined  as 
in item (ii) of \cite[Definition 3.1]{ANT}.

\item We say that $(J, E)$ is a {\it strong $n$--blur}, if it $(J, E)$ is an $n$--blur,  such that the complex 
$n$--blur  satisfies:
$$(\forall V_1,\ldots V_n, W_2,\ldots W_n\in J)(\forall T\in J)(\forall 2\leq i\leq n)
{\sf safe}(V_i,W_i,T).$$ 
\end{enumerate}
\end{definition}

\begin{theorem}\label{ANT}
For every $2<n<l<\omega$, there is an algebra $\B$  
in $\Nr_n\CA_l\cap \RCA_n$,  but is not strongly representable up to $l$ and $\omega$. In particular, $\B$ is  not strongly representable.
\end{theorem}
\begin{proof} We give an example of a blowing up and blurring a finite relation algebra $\R$  getting an infinite countable atomic ${\cal R}\in \sf RA$ such that that $\At\cal R$ is  weakly
but not strongly representable.
Furthermore $\cal R$ has an $n$ dimensional cylindric basis, and ${\sf Mat}_n(\At\cal \R)$ is a weakly but 
not strongly representable $\CA_n$ atom structure. This example is based on a generalization of the construction in \cite{ANT}. 
Our exposition of the construction in \cite{ANT} will be addressing an (abstract) finite relation algebra $\R$ having an $l$--blur in the sense of definition \cite[Definition 3.1]{ANT}, 
with $3\leq l\leq k<\omega$ and $k$ depending on $l$.  Occasionally we use the concrete 
Maddux algebra $\mathfrak{E}_k(2, 3)$ to make certain concepts more tangible.
Here $k$ is the number of non-identity atoms is concrete example of $\R$. 
In this algebra a triple $(a, b, c)$ of non--identity atoms is consistent $\iff$ $|\{a, b, c\}|\neq 1$, i.e 
only monochromatic triangles are forbidden. 

We use the notation in \cite{ANT}. Let $2<n\leq l<\omega$. One starts with a finite relation algebra $\R$ that has only representations, if any, on finite sets (bases), 
having an $l$--blur $(J, E)$ as in \cite[Definition 3.1]{ANT} recalled in definition \ref{strongblur}. 
After {\it blowing up and bluring $\R$}, by splitting each of its atoms into infinitely many, one gets 
an infinite atomic representable relation algebra ${\mathfrak Bb}(\R, J, E)$ \cite[p.73]{ANT}, whose atom structure $\bf At$ is weakly but not
strongly representable. The atom structure $\bf  At$ is not strongly representable, because $\R$ is {\it not blurred} in ${\sf Cm}\bf At$. 
The finite relation algebra $\R$ embeds into $\Cm\bf At$, so that a representation 
of $\Cm\bf At$, necessarily on
an infinite base, induces one of $\R$ on the same base, which is impossible.
The representability of ${\mathfrak Bb}(\R, J, E)$ depend on the properties of the $l$--blur,  which {\it blurs $\R$ in ${\mathfrak Bb}(\R, J, E)$}.
The set of blurs here, namely, $J$ is finite. In the case of $\mathfrak{E}_k(2, 3)$ used in \cite{ANT},  the set of blurs 
is the set of all subsets of non--identity atoms having the same size $l<\omega$, where $k=f(l)\geq l$ 
for some recursive function $f$ from $\omega\to \omega$, so that $k$ depends recursively on $l$. 
One (but not the only) way to define the {\it index blur} $E\subseteq {}^3\omega$ is as follows \cite[Theorem 3.1.1]{Sayed}:
$E(i,j,k)\iff (\exists p,q,r)(\{p,q,r\}=\{i,j,k\} \text { and } r-q=q-p.$
This is a concrete instance of an index blur as defined in \cite[Definition 3.1(iii)]{ANT} (recalled in definition \ref{strongblur} above), 
but defined uniformly, it does not depends on the blurs.
The underlying set of $\bf At$, the atom structure of ${\mathfrak Bb}(\R, J, E)$ is the following set consisting of triplets:
$At=\{(i, P, W): i\in \omega, P\in \At\R\sim \{\Id\}, W\in J\}\cup \{\Id\}$.
When $\R=\mathfrak{E}_k(2, 3)$ (some finite $k>0)$, composition  is defined by singling out the following (together with their Peircian transforms), 
as the consistent triples:
$(a, b, c)$ is consistent $\iff$ one of $a, b, c$ is $\sf Id$ and the other two are equal, or 
if $a=(i, P, S), b=(j, Q, Z), c= (k, R, W)$ 
$$S\cap Z\cap W\neq \emptyset \implies E(i,j,k)\&|\{P,Q,R\}|\neq 1.$$ 
(We are avoiding mononchromatic triangles).
That is if for $W\in J$,  $E^W=\{(i, P, W): i\in \omega, P\in W\},$
then $$(i, P, S); (j, Q, Z)=\bigcup\{E^W: S\cap Z\cap W=\emptyset\}$$
$$\bigcup \{(k, R, W): E(i,j,k), |\{P,Q,R\}|\neq 1\}.$$

More generally, for the $\R$ as postulated in the hypothesis, composition in $\bf At$ is defined as follow. First 
the index blur $E$ can be taken to be like above. 
Now 
the triple  $((i, P, S), (j, Q, Z), (k, R, W))$ in which no two entries are equal, is consistent
if either $S, Z, W$ are ${\it safe}$, briefly ${\sf safe}(S, Z, W)$, 
witness item  (4) in definition \ref{strongblur} (which vacuously hold if $S\cap Z\cap W=\emptyset$), 
or $E(i, j, k)$ and $P; Q\leq  R$ in $\R$. 
This generalizes the above definition of composition, 
because in $\mathfrak{E}_k(2, 3)$, the triple of non--identity atoms 
$(P, Q, R)$ is consistent $\iff$ they do not have the same colour $\iff$ $|\{P, Q, R\}|\neq 1.$
Having specified its atom structure,  its timely to 
specfiy the relation algebra ${\mathfrak Bb}(\R, J, E)\subseteq \Cm{\bf At}$.
The relation algebra ${\mathfrak Bb}(\R, J, E)$ is $\Tm\bf At$ (the term algebra).
Its  universe is the set $\{X\subseteq H\cup \{\Id\}: X\cap E^W\in {\sf Cof}(E^W), \text{ for all } W\in J\}$, where 
${\sf Cof}(E^{W})$ denotes the set of co--finite subsets of $E^{W}$, 
that is subsets of $E^W$ whose complement is infinite,  with $E^W$ as defined above. The relation algebra 
operations are lifted from $\bf At$ the usual way.
The algebra  ${\mathfrak Bb}(\R, J, E)$ is proved to be representable \cite{ANT} as shown next.
For brevity, denote ${\mathfrak Bb}(\R, J, E)$ by $\cal R$, and its domain by $R$.
For $a\in \bf At$, and $W\in J,$  set
$U^a=\{X\in R: a\in X\}\text { and } U^{W}=\{X\in R: |X\cap E^W|\geq \omega\}.$
Then the principal ultrafilters of $\cal R$ are exactly $U^a$, $a\in H$ and $U^W$
are non-principal ultrafilters for $W\in J$ when $E^W$ is infinite.
Let  $J'=\{W\in J: |E^W|\geq \omega\},$
and let ${\sf Uf}=\{U^a: a\in F\}\cup \{U^W: W\in J'\}.$
${\sf Uf}$ is the set of ultrafilters of $\cal R$ which is used as colours
to represent $\cal R$, cf. \cite[pp. 75-77]{ANT}. The representation is 
built from coloured graphs whose edges are labelled 
by elements in ${\sf Uf}$   in a fairly standard step--by--step construction.

Now we show  why the \de\ completion $\Cm \bf At$ is {\it not} representable. 
For $P\in I$, let $H^P=\{(i, P,W): i\in \omega, W\in J, P\in W\}$.
Let  $P_1=\{H^P: P\in I\}$ and $P_2=\{E^W: W\in J\}$.   These are two partitions of $At$. 
The partition $P_2$  was used to {\it represent},
${\mathfrak Bb}(\R, J, E)$, in the sense that the tenary relation corresponding to composition 
was defined on $\bf At$, in a such a way so that the singletons generate the partition
$(E^W: W\in J)$ up to ``finite deviations." 
The  partition $P_1$ will now be used to show that $\Cm({\mathfrak Bb}(\R, J, E))=\Cm (\bf At)$ 
is {\it not }  representable.   This follows by observing that 
omposition restricted to $P_1$ satisfies: $ H^P;H^Q=\bigcup \{H^Z: Z;P\leq Q \text { in } \R\}$
which means that $\R$ embeds into the complex algebra 
$\Cm \bf At$ prohibiting its representability, 
because $\R$ allows only representations having 
a finite base.
So far we have been dealing with relation algebras. The construction lifts to higher dimensions expressed in $\CA_n$s, $2<n<\omega$. 
as shown next. 
Let $\R$ be as in the hypothesis. 
Let $3<n\leq l$. We blow up and blur $\R$. $\R$ is blown up by splitting all of the atoms each to infinitely many
defining an (infinite atoms) structure $\bf At$.
$\R$ is blurred by using a finite set of blurs (or colours) $J$. 
The term algebra ${\Bb}(\R, J, E)$) over $\bf At$, 
 is representable using the finite number of blurs. Such blurs are basically non--principal ultrafilters; they are used as colours together 
with the principal ultrafilters (the atoms) to represent completely the canonical extension of ${\Bb}(\R, J, E)$. 
%This representation is implemented in  step-by-step manner, and in fact this step by step construction 
%adopted in \cite{ANT} {\it completely represents the canonical
%extension} of ${{\sf split}}(\R, J, E)$.
Because $(J, E)$ is a complex set of $l$--blurs, this atom structure has an $l$--dimensional cylindric basis, 
namely, ${\bf At}_{ca}={\sf Mat}_l(\bf At)$. The resulting $l$--dimensional cylindric term algebra $\Tm{\sf Mat}_l(\bf At)$, 
and an algebra $\C$ having atom structure ${\bf At}_{ca}$ (denoted in \cite{ANT} by 
$\mathfrak{Bb}_l(\R, J, E)$) such that $\Tm{\sf Mat}_l({\bf At})\subseteq \C\ \subseteq \Cm{\sf Mat}_l(\bf At)$ 
is shown to be  representable.  
Assume that the $m$--blur $(J, E)$ is strong, then by definition $(J, E)$ is a strong  $j$ blur for all $n\leq j\leq m$.
Furthermore,  by \cite[item (3) pp. 80]{ANT},  
$\Bb(\R, J, E)=\Ra{\Bb}_j(\R, J, E))$ 
and ${\Bb}_j(\R, J, E)\cong \mathfrak{Nr}_j{{\Bb}}_m(\R, J, E)$. 
\end{proof} 
${\sf LCA}_n$ denotes the elementary class of ${\sf RCA}_n$s satisfying the Lyndon conditions \cite[Definition 3.5.1]{HHbook2}.
\begin{theorem}\label{square} Let $2<n<m\leq \omega$.
Then ${\bf El}\Nr_n\CA_{\omega}\cap {\bf At}\subsetneq {\sf LCA}_n$. Furthermore, 
for any elementary class $\sf K$ between ${\bf El}\Nr_n\CA_{\omega}\cap \bf At$ and ${\sf LCA}_n$, ${\sf RCA}_n$ is generated by 
$\At\sf K$.
\end{theorem}
\begin{proof} It suffices to show that ${\sf Nr}_n\CA_{\omega}\cap {\bf At}\subseteq {\sf LCA}_n$, since the last class is elementary. This follows from Lemma \ref{n}, 
since if $\A\in \Nr_n\CA_{\omega}$ is atomic, then \pe\ has a \ws\ in $\bold G^{\omega}(\At\A)$, hence in $G_{\omega}(\At\A)$, {\it a fortiori}, \pe\ has a \ws\ 
in $G_k(\At\A)$ for all $k<\omega$, so (by definition) $\A\in {\sf LCA}_n$. 
To show strictness of the last inclusion, let $V={}^n\Q$ and let ${\A}\in {\sf Cs}_n$ have universe $\wp(V)$.
Then $\A\in {\sf Nr}_{n}\CA_{\omega}$.  
Let 
$y=\{s\in V: s_0+1=\sum_{i>0} s_i\}$ and ${\B}=\Sg^{\A}(\{y\}\cup X)$, where $X=\{\{s\}: s\in V\}$. 
Now $\B$ and $\A$ having same top element $V$, share the same atom structure, namely, the singletons, so 
$\Cm\At\B=\A$. Furthermore, plainly $\A, \B\in {\sf CRCA}_n$. 
So $\B\in {\sf CRCA}_n\subseteq {\sf LCA}_n$, 
and as proved in \cite{SL}, $\B\notin {\bf  El}\Nr_{n}{\sf CA}_{n+1}$, hence $\B$ witnesses the required strict inclusion. 

Now we show that $\At{\bf El}\Nr_n\CA_{\omega}$ generates $\RCA_n$.
Let ${\sf FCs}_n$ denote the class of {\it full} $\Cs_n$s, that is ${\sf Cs}_n$s  
having universe $\wp(^nU)$
($U$ non--empty set). 
First we show that ${\sf FCs}_n\subseteq \Cm\At\Nr_n\CA_{\omega}$.
Let $\A\in {\sf FCs}_n$.  Then $\A\in \Nr_n\CA_{\omega}\cap \bf At$, hence $\At\A\in \At\Nr_n\CA_{\omega}$ 
and $\A=\Cm\At\A\in \Cm\At\Nr_n\CA_{\omega}$.
The required now follows from the following chain of inclusions: 
$\RCA_n={\bf SP}{\sf FCs}_n\subseteq {\bf SP}\Cm\At(\Nr_n\CA_{\omega})\subseteq {\bf SP}\Cm\At({\bf El}\Nr_n\CA_{\omega})\subseteq 
{\bf SP}\Cm \At{\sf K}\subseteq {\bf SP}\Cm{\sf LCAS}_n\subseteq  {\sf RCA}_n,$ where $\sf K$ is given above.

\end{proof}

Let $2<n\leq l\leq m\leq \omega$.
Denote the class of $\CA_n$s having the complex $\bold O$ neat embedding property up to $m$ by ${\sf CNPCA}_{n,m}^{\bold O}$,
and let ${\sf RCA}_{n,m}^{\bold O}:={\sf CNPCA}_{n,m}^{\bold O}\cap {\sf RCA}_n$.
Denote the class of strongly representable $\CA_n$s up to $l$ and $m$ by ${\sf RCA}_n^{l,m}$.
Observe that ${\sf RCA}_n^{n, m}={\sf RCA}_{n, m}^{\bold S}$ and that when $m=\omega$ both classes coincide with the class of strongly representable $\CA_n$s.
For a class $\bold K$ of $\sf BAO$s, $\bold K\cap {\sf Count}$ 
denotes the class of countable algebras in $\bold K$,  and recall that $\bold K\cap \bf At$ denotes the class of atomic algebras in $\bold K$. 
\begin{theorem}\label{main2} Let $2<n\leq l<m\leq \omega$ and $\bold O\in \{\bf S, S_c, S_d, I\}$. 
Then the following hold:
\begin{enumerate}
\item ${\sf RCA}_{n,m}^{\bold O}\subseteq {\sf RCA}_{n,l}^{\bold O}$ 
and  ${\sf RCA}_{n,l}^{\bf I}\subseteq {\sf RCA}_{n, l}^{\bold S_d}\subseteq {\sf RCA}_{n, l}^{\bold S_c}\subseteq {\sf RCA}_{n,l}^{\bold S}$. 
The  last inclusion is proper for $l\geq n+3$,
\item For $\bold O\in \{\bf S, S_c, S_d\}$, ${\sf CNPCA}_{n,l}^{\bold O}\subseteq \bold O\Nr_n\CA_l$ (that is the complex $\bold O$ neat embedding property is stronger than the $\bold O$ 
neat embedding property),
and for $\bold O=\bold S$, the inclusion is proper for $l\geq n+3$. But for $\bold O=\bold I$, ${\sf CNPCA}_{n,l}^{\bold I}\nsubseteq \Nr_n\CA_l$ 
(so the complex $\bold I$  neat embedding property {\it does not imply} 
the $\bold I$ neat embedding property), 
\item If $\A$ is finite, then $\A\in {\sf CNPCA}_{n,l}^{\bold O}\iff \A\in \bold O\Nr_n\CA_l$ 
and $\A\in {\sf RCA}_{n, l}^{\bold O}\iff \A\in \RCA_n\cap \bold O\Nr_n\CA_l$.
Furthermore, for any positive $k$, ${\sf CNPCA}_{n, n+k+1}^{\bold O}\subsetneq {\sf CNPCA}_{n, {n+k}}^{\bold O},$
and finally ${\sf CNPCA}_{n,\omega}^{\bold O}\subsetneq \RCA_n,$
\item  $(\exists \A\in \RCA_n\cap {\bf At}\sim {\sf CNPCA}_{n,l}^{\bold S})\implies \bold S{\sf Nr}_n\CA_k$ is not 
atom--canonical for all $k\geq l$. In particular, $\bold S{\sf Nr}_n\CA_k$ is not atom--canonical for all $k\geq n+3$,
\item If $\bold S\Nr_n\CA_l$ is atom--canonical, then $\RCA_{n,l}^{\bold S}$ is first order definable. 
There exists a finite $k>n+1$, such that ${\sf RCA}_{n, k}^{\bold S}$ is not first order definable.
\item Let $2<n<l\leq \omega$. Then $\RCA_n^{l, \omega}\cap {\sf Count}\neq \emptyset\iff l<\omega$.

\end{enumerate} 
\end{theorem}
\begin{proof}
 
(1): The inclusions in the first item are by definition. To show the strictness of the last inclusion, we proceed in this way. 
We show that there an $\sf RCA_n$ with countably many atoms 
outside $\bold S_c\Nr_n\CA_{n+3}$.
Take the a rainbow--like $\CA_n$, call it $\C$, based on the ordered structure $\Z$ and $\N$.
The reds ${\sf R}$ is the set $\{\r_{ij}: i<j<\omega(=\N)\}$ and the green colours used 
constitute the set $\{\g_i:1\leq i <n-1\}\cup \{\g_0^i: i\in \Z\}$. 
In complete coloured graphs the forbidden triples are like 
the usual rainbow constructions based on $\Z$ and $\N$,   
but now 
the triple  $(\g^i_0, \g^j_0, \r_{kl})$ is also forbidden if $\{(i, k), (j, l)\}$ is not an order preserving partial function from
$\Z\to\N$.
It can be shown that \pa\ has a \ws\ in the graph version of the game 
$\bold G^{n+3}(\At\C)$ played on coloured graphs \cite{HH}.
The rough idea here, is that, as is the case with \ws's of \pa\ in rainbow constructions, 
\pa\ bombards \pe\ with cones having distinct green tints demanding a red label from \pe\ to appexes of succesive cones.
The number of nodes are limited but \pa\ has the option to re-use them, so this process will not end after finitely many rounds.
The added order preserving condition relating two greens and a red, forces \pe\ to choose red labels, one of whose indices form a decreasing 
sequence in $\N$.  In $\omega$ many rounds \pa\ 
forces a win, 
so $\C\notin \bold S_c{\sf Nr}_n\CA_{n+3}$.
%The detailed proof is in \cite[Theorem 5.12]{mlq}.
More rigorously, \pa\ plays as follows: In the initial round \pa\ plays a graph $M$ with nodes $0,1,\ldots, n-1$ such that $M(i,j)=\w_0$
for $i<j<n-1$
and $M(i, n-1)=\g_i$
$(i=1, \ldots, n-2)$, $M(0, n-1)=\g_0^0$ and $M(0,1,\ldots, n-2)=\y_{\Z}$. This is a $0$ cone.
In the following move \pa\ chooses the base  of the cone $(0,\ldots, n-2)$ and demands a node $n$
with $M_2(i,n)=\g_i$ $(i=1,\ldots, n-2)$, and $M_2(0,n)=\g_0^{-1}.$
\pe\ must choose a label for the edge $(n+1,n)$ of $M_2$. It must be a red atom $r_{mk}$, $m, k\in \N$. Since $-1<0$, then by the `order preserving' condition
we have $m<k$.
In the next move \pa\ plays the face $(0, \ldots, n-2)$ and demands a node $n+1$, with $M_3(i,n)=\g_i$ $(i=1,\ldots, n-2)$,
such that  $M_3(0, n+2)=\g_0^{-2}$.
Then $M_3(n+1,n)$ and $M_3(n+1, n-1)$ both being red, the indices must match.
$M_3(n+1,n)=r_{lk}$ and $M_3(n+1, r-1)=r_{km}$ with $l<m\in \N$.
In the next round \pa\ plays $(0,1,\ldots n-2)$ and re-uses the node $2$ such that $M_4(0,2)=\g_0^{-3}$.
This time we have $M_4(n,n-1)=\r_{jl}$ for some $j<l<m\in \N$.
Continuing in this manner leads to a decreasing
sequence in $\N$. We have proved the required.
Since $\Cm\At\C=\C$ 
and $\C\notin \bold S_c\Nr_n\CA_{n+3}$ we are done.

(2):  Let $\bold O\in \{\bf S, S_c, S_d\}$. 
If $\Cm\At\A\in \bold O\Nr_n\CA_l$, then $\A\subseteq_d \Cm\At\A$, so $\A\in \bold S_d\bold O\Nr_n\CA_l\subseteq \bold O\Nr_n\CA_l$.
This proves the first part. 
The strictness of the last inclusion follows from Theorem \ref{can} since the atomic countable algebra $\A$ 
constructed in {\it op.cit} is in $\RCA_n$, but $\Cm\At\A$ does not have an $n+3$-square representation, least is in $\bold S\Nr_n\CA_{l}$ for any $l\geq n+3$.
For the last non--inclusion in item (2), we use the set algebras $\A$ and $\E$ in Theorem \ref{square}.
Now $\B\subseteq_d \A$, $\A\in {\sf Cs}_n$, and clearly $\Cm\At\B=\A(\in \Nr_n\CA_{\omega}$).   

\item Follows by definition observing that if $\A$ is finite then 
$\A=\Cm\At\A$. The strictness of the first inclusion follows from the construction in \cite{t} 
where it shown that for any positive $k$, 
there is a {\it finite algebra $\A$} in 
$\bold \Nr_n\CA_{n+k}\sim \bold S\Nr_n\CA_{n+k+1}$.  
The inclusion ${\sf CNPCA}_{n,\omega}^{\bold O}\subseteq \RCA_n$ holds because if $\B\in {\sf CNPCA}_{n,\omega}^{\bold O}$, then 
$\B\subseteq \Cm\At\B\in \bold O\Nr_n\CA_{\omega}\subseteq \RCA_n$.
The $\A$ used in the last item of theorem \ref{can} witnesses  the strictness of the last inclusion proving the last required in this item.

(3): Follows from the definition and the construction used above.

(4):  Follows from that ${\bf S}\Nr_n\CA_l$ is canonical. So if it is atom--canonical too, 
then $\At(\bold S\Nr_n\CA_{l})=\{\F: \Cm\F\in \bold S\Nr_n\CA_l\},$ 
the former class is elementary \cite[Theorem 2.84]{HHbook}, and the last class is elementray $\iff \RCA_{n,l}^{\bold S}$ is 
elementary.
Non--elementarity follows from  \cite[Corollary 3.7.2]{HHbook2} where it is proved that ${\sf RCA}_{n, \omega}^{\bold S}$ is not elementary, together with the fact 
that  $\bigcap_{n<k<\omega}{\bold S}\Nr_n\CA_{k}=\RCA_n$.
In more detail, let $\A_i$ be the sequence of strongly representable $\CA_n$s with $\Cm\At\A_i=\A_i$
and $\A=\Pi_{i/U}\A_i$ is not strongly representable.
Hence $\Cm\At\A\notin \bold S\Nr_n\CA_{\omega}=\bigcap_{i\in \omega}\bold S\Nr_n\CA_{n+i}$, 
so $\Cm\At\A\notin \bold S\Nr_n\K_{l}$ for all $l>k$, for some $k\in \omega$, $k>n$. 
But for each such $l$, $\A_i\in \bold S\Nr_n\CA_l(\supseteq {\sf RCA}_n$), 
so $\A_i$ is a sequence of algebras such that $\Cm\At\A_i=\A_i\in \bold S\Nr_n\CA_{l}$, but 
$\Cm(\At(\Pi_{i/U}\A_i))=\Cm\At\A\notin \bold S\Nr_n\CA_l$, for all $l\geq k$.
That $k$ has to be strictly greater than $n+1$, follows because 
$\bold S\Nr_n\CA_{n+1}$ is atom--canonical. 

(5):  $\Longleftarrow$: Let $l<\omega$. Then the required follows from Theorem \ref{can}  namely, there exists a countable 
$\A\in \Nr_n\CA_l\cap \RCA_n$ such  that $\Cm\At\A\notin \RCA_n$.
Now we prove $\implies$:  
Assume for contradiction that there is an $\A\in {\sf RCA}_n^{\omega, \omega}\cap \sf Count$. 
Then by definition  $\A\in \Nr_n\CA_{\omega}$, so $\A\in \CRCA_n$.  But this complete representation induces a(n ordinary) 
representation of $\Cm\At\A$ 
which is a 
contradiction. 

\end{proof}
\section{Complete and other forms of representations}

\begin{theorem}\label{complete} Let $\alpha$ be any countable ordinal (possibly infinite) and $\A\in \CA_{\alpha}$.
If $\A$ is  atomic with countably many atoms,
then $\A$ is completely representable $\iff \A\in \bold S_c\Nr_{\alpha}\CA_{\alpha}\cap \bf At$. The implication $\implies$ holds with no restriction on the cardinality of atoms. 
%The implication $\implies$ holds without any restriction on 
%the cardinality of the atoms. 
\end{theorem}
\begin{proof} Assume that  $\A\subseteq_c \mathfrak{N}r_{\alpha} \D$.
We can assume that $\A$ is countable and $\D\in \sf Dc_{\alpha+\omega}$.
%Let $\B=\Sg^{\D}\A$.  Then  $\B$ is countable, too,  
%$\A\subseteq \mathfrak{Nr}_n\B$ and $\B\in \sf DKc_{\alpha+\omega}$. 
Now we use exactly the argument \cite[Theorem 3.2.4]{Sayed}, replacing $\Fm_T$ in {\it op.cit} by $\B$. Omitting the one non--principal type of co--atoms,
we get the required complete representation.
%We prove the other implication. Here we do not require that $\A$ has countably many atoms. 
%The proof is very similar to the proof of \cite[Theorem 29]{r}, except that we deal with weak set algebras when dealing with infinite dimensions. 
%Furthermore, we have to check that the dilations defined during the proof are not merely atomic, but also 
%{\it completely additive}. 
Assume that $\Mo$ is the base of a complete representation of $\A$, whose
unit is a weak generalized space,
that is, $1^{\Mo}=\bigcup {}^\alpha U_i^{(p_i)}$ $p_i\in {}^{\alpha}U_i$, where $^{\alpha}U_i^{(p_i)}\cap {}^{\alpha}U_j^{(p_j)}=\emptyset$ for distinct $i$ and $j$, in some
index set $I$, that is, we have an isomorphism $t:\B\to \C$, where $\C\in \sf Gs_{\alpha}$ 
has unit $1^{\Mo}$, and $t$ preserves arbitrary meets carrying
them to set--theoretic intersections.
For $i\in I$, let $E_i={}^{\alpha}U_i^{(p_i)}$. Take  $f_i\in {}^{\alpha+\omega}U_i^{(q_i)}$ where $q_i\upharpoonright \alpha=p_i$
and let $W_i=\{f\in  {}^{\alpha+\omega}U_i^{(q_i)}: |\{k\in \alpha+\omega: f(k)\neq f_i(k)\}|<\omega\}$.
Let ${\C}_i=\wp(W_i)$. Then $\C_i$ is atomic; indeed the atoms are the singletons.
Let $x\in \mathfrak{Nr}_{\alpha}\C_i$, that is ${\sf c}_ix=x$ for all $\alpha\leq i<\alpha+\omega$.
Now if  $f\in x$ and $g\in W_i$ satisfy $g(k)=f(k) $ for all $k<\alpha$, then $g\in x$.
Hence $\mathfrak{Nr}_{\alpha}\C_i$
is atomic;  its atoms are $\{g\in W_i:  \{g(i):i<\alpha\}\subseteq U_i\}.$
Define $h_i: \A\to \mathfrak{Nr}_{\alpha}\C_i$ by
$h_i(a)=\{f\in W_i: \exists a'\in \At\A, a'\leq a;  (f(i): i<\alpha)\in t(a')\}.$
Let $\D=\bold P _i \C_i$. Let $\pi_i:\D\to \C_i$ be the $i$th projection map.
Now clearly  $\D$ is atomic, because it is a product of atomic algebras,
and its atoms are $(\pi_i(\beta): \beta\in \At(\C_i))$.  
Now  $\A$ embeds into $\mathfrak{Nr}_{\alpha}\D$ via $J:a\mapsto (\pi_i(a) :i\in I)$. If $x\in \mathfrak{Nr}_{\alpha}\D$,
then for each $i$, we have $\pi_i(x)\in \mathfrak{Nr}_{\alpha}\C_i$, and if $x$
is non--zero, then $\pi_i(x)\neq 0$. By atomicity of $\C_i$, there is an $\alpha$--ary tuple $y$, such that
$\{g\in W_i: g(k)=y_k\}\subseteq \pi_i(x)$. It follows that there is an atom
of $b\in \A$, such that  $x\cdot  J(b)\neq 0$, and so the embedding is atomic, hence complete.
We have shown that $\A\in \bold S_c\Nr_{\alpha}\CA_{\alpha+ \omega}$
and we are done.  
\end{proof}

Fix $2<n<\omega$. Call an atomic $\A\in \CA_n$ {\it weakly (strongly) representable} $\iff \At\A$ is weakly (strongly) representable.
Let  ${\sf WRCA}_n$ (${\sf SRCA}_n$) denote the class of all such $\CA_n$s, respectively. 
Then the class ${\sf SRCA}_n$  is not elementary and 
${\sf LCA}_n\subsetneq {\sf SRCA}_n\subsetneq {\sf WRCA}_n$ \cite{HHbook2}. 
\begin{theorem}\label{finalresult} Let $2<n<\omega$. Then the following hold: 
\begin{enumerate}
\item $\bold S_c\Nr_n\CA_{\omega}\cap {\sf Count}={\sf CRCA}_n\cap {\sf Count}$,
and ${\bf El}\bold S_c\Nr_n\CA_{\omega}\cap {\bf At}={\sf LCA}_n$, 
\item $\bold S\Nr_n\CA_{\omega}\cap {\bf At}= {\sf WRCA}_n$, and ${\bf P El S}_c\Nr_n\CA_{\omega}\cap {\bf At}\subseteq {\sf SRCA}_n$.
\end{enumerate}
\end{theorem}
\begin{proof}
For the first required one uses \cite[Theorem 5.3.6]{Sayedneat}. 
For the second required, show that ${\sf LCA}_n={\bf El}{\sf CRCA}_n={\bf El}(\bold S_c{\sf Nr}_n\CA_{\omega}\cap {\bf At}$). 
Assume that $\A\in {\sf LCA}_n$.
Then, by definition,  for all $k<\omega$, \pe\ has a \ws\ in $G_k(\At\A)$. Using ultrapowers followed by an elementary chain argument like in  \cite[Theorem 3.3.5]{HHbook2},   \pe\ has a \ws\ in
$G_{\omega}(\At\B)$ for some countable $\B\equiv \A$, and so by \cite[Theorem 3.3.3]{HHbook2} $\B$ is completely representable.
Thus $\A\in {\bf El}{\sf CRCA}_n$. One shows that 
${\bf El}(\bold S_c{\sf Nr}_n\CA_{\omega}\cap {\bf At})\subseteq {\sf LCA}_n$ exactly like in item (1) of Theorem \ref{square}.
So ${\sf LCA}_n={\bf El}{\sf CRCA}_n\subseteq {\bf El}(\bold S_c{\sf Nr}_n\CA_{\omega}\cap {\bf At})\subseteq {\sf LCA}_{n}$, 
and we are done. 
The first part of item (2) follows from the definition and the last part follows from that ${\sf LCA}_n\subseteq {\sf SRCA}_n$, and that (it is easy to check that) 
${\sf SRCA}_n$ is closed under $\bold P$.
\end{proof}

\begin{theorem}
For $2<n<\omega$. Then $\CRCA_n$ is not elementary \cite{HH}. 
Furthermore, ${\sf CRCA}_{n} \subseteq \bold S_c{\Nr}_{n}(\CA_{\omega}\cap {\bf At})\cap {\bf At}\subseteq \bold S_c{\Nr}_{n}\CA_{\omega}\cap \bf At$. 
At least two of the previous three  classes are distinct but the elementary closure of each coincides with 
${\sf LCA}_n$. Furthermore, all three classes coincide on the class of atomic  algebras having countably many atoms.
\end{theorem}
%\begin{proof}
%One uses the ideas in \cite{bsl} replacing 
%$\omega$ and $\omega_1$ by $\kappa$ and $2^{\kappa}$, respectively, constructing $\C$ from a relation algebra. 
%The resulting (new) relation algebra $\R$ has an $\omega$ 
%dimensional amalgamation class $S$, cf. \cite[Lemma 3]{bsl}.
%Using the notation in \cite[Lemma 6]{bsl}, let $\C$ be the subalgebra of $\Ca(S)$ generated by $X'$; the latter is defined just before the lemma.
%Then  $\R = \Ra(\C)$, cf. \cite[Lemmata 6, 7]{bsl}, but $\R$ has no complete representation \cite[Lemma 2]{bsl}.
%Then $\mathfrak{Nr}_n\C$ ($2<n<\omega$) is atomic, but has no complete representation and by the first item of Theorem \ref{square}, $\mathfrak{Nr}_n\C\in {\sf LCA}_n$. 
%\end{proof}
\begin{proof} We use the following uncountable version of Ramsey's theorem due to
Erdos and Rado:
If $r\geq 2$ is finite, $k$  an infinite cardinal, then
$exp_r(k)^+\to (k^+)_k^{r+1}$
where $exp_0(k)=k$ and inductively $exp_{r+1}(k)=2^{exp_r(k)}$.
The above partition symbol describes the following statement. If $f$ is a coloring of the $r+1$
element subsets of a set of cardinality $exp_r(k)^+$
in $k$ many colors, then there is a homogeneous set of cardinality $k^+$
(a set, all whose $r+1$ element subsets get the same $f$-value).
Let $\kappa$ be a given infinite cardinal. We shall construct an atomless algebra $\C\in \CA_{\omega}$ such that for all $n<\omega$, $\Nrr_n\C$ 
is atomic having uncountably many atoms, but lacks a complete representation.
An application of Lemma \ref{n} will finish the proof.  
We use a simplified more basic version of a rainbow construction where only 
the two predominent  colours, namely, the reds and blues are available. 
The algebra $\C$ will be constructed from a relation algebra possesing an $\omega$-dimensional cylindric basis.
To define the relation algebra we specify its atoms and the forbidden triples of atoms. The atoms are $\Id, \; \g_0^i:i<2^{\kappa}$ and $\r_j:1\leq j<
\kappa$, all symmetric.  The forbidden triples of atoms are all
permutations of $({\sf Id}, x, y)$ for $x \neq y$, \/$(\r_j, \r_j, \r_j)$ for
$1\leq j<\kappa$ and $(\g_0^i, \g_0^{i'}, \g_0^{i^*})$ for $i, i',
i^*<2^{\kappa}.$ 
Write $\g_0$ for $\set{\g_0^i:i<2^{\kappa}}$ and $\r_+$ for
$\set{\r_j:1\leq j<\kappa}$. Call this atom
structure $\alpha$.  
Consider the term algebra $\R$ defined to be the subalgebra of the complex algebra of this atom structure generated by the atoms.
We claim that $\R$, as a relation algebra,  has no complete representation, hence any algebra sharing this 
atom structure is not completely representable, too. 
Indeed, it is easy to show that if $\A$ and $\B$ 
are atomic relation algebras sharing the same atom structure, so that $\At\A=\At\B$, then $\A$ is completely representable $\iff$ $\B$ is completely representable.

Assume for contradiction that $\R$ has a complete representation $\Mo$.  Let $x, y$ be points in the
representation with $\Mo \models \r_1(x, y)$.  For each $i< 2^{\kappa}$, there is a
point $z_i \in \Mo$ such that $\Mo \models \g_0^i(x, z_i) \wedge \r_1(z_i, y)$.
Let $Z = \set{z_i:i<2^{\kappa}}$.  Within $Z$, each edge is labelled by one of the $\kappa$ atoms in
$\r_+$.  The Erdos-Rado theorem forces the existence of three points
$z^1, z^2, z^3 \in Z$ such that $\Mo \models \r_j(z^1, z^2) \wedge \r_j(z^2, z^3)
\wedge \r_j(z^3, z_1)$, for some single $j<\kappa$.  This contradicts the
definition of composition in $\R$ (since we avoided monochromatic triangles).
Let $S$ be the set of all atomic $\R$-networks $N$ with nodes
$\omega$ such that $\{\r_i: 1\leq i<\kappa: \r_i \text{ is the label
of an edge in $N$}\}$ is finite.
Then it is straightforward to show $S$ is an amalgamation class, that is for all $M, N
\in S$ if $M \equiv_{ij} N$ then there is $L \in S$ with
$M \equiv_i L \equiv_j N$, witness \cite[Definition 12.8]{HHbook} for notation.
%We have $S$ is symmetric, that is, if $N\in S$ and $\theta:\omega\to \omega$ is a finitary function, in the sense
%that $\{i\in \omega: \theta(i)\neq i\}$ is finite, then $N\theta$ is in $S$. It follows that the complex
%algebra $\Ca(S)\in \QEA_\omega$.
Now let $X$ be the set of finite $\R$-networks $N$ with nodes
$\subseteq\kappa$ such that:

1. each edge of $N$ is either (a) an atom of
$\R$ or (b) a cofinite subset of $\r_+=\set{\r_j:1\leq j<\kappa}$ or (c)
a cofinite subset of $\g_0=\set{\g_0^i:i<2^{\kappa}}$ and

2.   $N$ is `triangle-closed', i.e. for all $l, m, n \in \nodes(N)$ we
have $N(l, n) \leq N(l,m);N(m,n)$.  That means if an edge $(l,m)$ is
labelled by $\sf Id$ then $N(l,n)= N(m,n)$ and if $N(l,m), N(m,n) \leq
\g_0$ then $N(l,n)\cdot \g_0 = 0$ and if $N(l,m)=N(m,n) =
\r_j$ (some $1\leq j<\omega$) then $N(l,n)\cdot \r_j = 0$.

For $N\in X$ let $\widehat{N}\in\Ca(S)$ be defined by
$$\set{L\in S: L(m,n)\leq
N(m,n) \mbox{ for } m,n\in \nodes(N)}.$$
For $i\in \omega$, let $N\restr{-i}$ be the subgraph of $N$ obtained by deleting the node $i$.
Then if $N\in X, \; i<\omega$ then $\widehat{\cyl i N} =
\widehat{N\restr{-i}}$.
The inclusion $\widehat{\cyl i N} \subseteq (\widehat{N\restr{-i})}$ is clear.
Conversely, let $L \in \widehat{(N\restr{-i})}$.  We seek $M \equiv_i L$ with
$M\in \widehat{N}$.  This will prove that $L \in \widehat{\cyl i N}$, as required.
Since $L\in S$ the set $T = \set{\r_i \notin L}$ is infinite.  Let $T$
be the disjoint union of two infinite sets $Y \cup Y'$, say.  To
define the $\omega$-network $M$ we must define the labels of all edges
involving the node $i$ (other labels are given by $M\equiv_i L$).  We
define these labels by enumerating the edges and labeling them one at
a time.  So let $j \neq i < \kappa$.  Suppose $j\in \nodes(N)$.  We
must choose $M(i,j) \leq N(i,j)$.  If $N(i,j)$ is an atom then of
course $M(i,j)=N(i,j)$.  Since $N$ is finite, this defines only
finitely many labels of $M$.  If $N(i,j)$ is a cofinite subset of
$\g_0$ then we let $M(i,j)$ be an arbitrary atom in $N(i,j)$.  And if
$N(i,j)$ is a cofinite subset of $\r_+$ then let $M(i,j)$ be an element
of $N(i,j)\cap Y$ which has not been used as the label of any edge of
$M$ which has already been chosen (possible, since at each stage only
finitely many have been chosen so far).  If $j\notin \nodes(N)$ then we
can let $M(i,j)= \r_k \in Y$ some $1\leq k < \kappa$ such that no edge of $M$
has already been labelled by $\r_k$.  It is not hard to check that each
triangle of $M$ is consistent (we have avoided all monochromatic
triangles) and clearly $M\in \widehat{N}$ and $M\equiv_i L$.  The labeling avoided all
but finitely many elements of $Y'$, so $M\in S$. So
$\widehat{(N\restr{-i})} \subseteq \widehat{\cyl i N}$.
Now let $\widehat{X} = \set{\widehat{N}:N\in X} \subseteq \Ca(S)$.
We claim that the subalgebra of $\Ca(S)$ generated by $\widehat{X}$ is simply obtained from
$\widehat{X}$ by closing under finite unions.
Clearly all these finite unions are generated by $\widehat{X}$.  We must show
that the set of finite unions of $\widehat{X}$ is closed under all cylindric
operations.  Closure under unions is given.  For $\widehat{N}\in X$ we have
$-\widehat{N} = \bigcup_{m,n\in \nodes(N)}\widehat{N_{mn}}$ where $N_{mn}$ is a network
with nodes $\set{m,n}$ and labeling $N_{mn}(m,n) = -N(m,n)$. $N_{mn}$
may not belong to $X$ but it is equivalent to a union of at most finitely many
members of $\widehat{X}$.  The diagonal $\diag ij \in\Ca(S)$ is equal to $\widehat{N}$
where $N$ is a network with nodes $\set{i,j}$ and labeling
$N(i,j)=\sf Id$.  Closure under cylindrification is given.
Let $\C$ be the subalgebra of $\Ca(S)$ generated by $\widehat{X}$.
Then $\R = \mathfrak{Ra}(\C)$.
To see why, each element of $\R$ is a union of a finite number of atoms,
possibly a co--finite subset of $\g_0$ and possibly a co--finite subset
of $\r_+$.  Clearly $\R\subseteq\mathfrak{Ra}(\C)$.  Conversely, each element
$z \in \mathfrak{Ra}(\C)$ is a finite union $\bigcup_{N\in F}\widehat{N}$, for some
finite subset $F$ of $X$, satisfying $\cyl i z = z$, for $i > 1$. Let $i_0,
\ldots, i_k$ be an enumeration of all the nodes, other than $0$ and
$1$, that occur as nodes of networks in $F$.  Then, $\cyl
{i_0} \ldots
\cyl {i_k}z = \bigcup_{N\in F} \cyl {i_0} \ldots
\cyl {i_k}\widehat{N} = \bigcup_{N\in F} \widehat{(N\restr{\set{0,1}})} \in \R$.  So $\mathfrak{Ra}(\C)
\subseteq \R$.
Thus $\R$ is relation algebra reduct of $\C\in\CA_\omega$ but has no complete representation.
Let $n>2$. Let $\B=\Nrr_n \C$. Then
Thus $\R$ is relation algebra reduct of $\C\in\CA_\omega$ but has no complete representation.
Let $n>2$. Let $\B=\Nrr_n \C$. Then
$\B\in {\sf Nr}_n\CA_{\omega}$, is atomic, but has no complete representation for plainly a complete representation of $\B$ induces one of $\R$. 
In fact, because $\B$  is generated by its two dimensional elements,
and its dimension is at least three, its
$\Df$ reduct is not completely representable.
%\cite[Proposition 4.10]{AU}.
We show that the $\omega$--dilation $\C$ is atomless. 
For any $N\in X$, we can add an extra node 
extending
$N$ to $M$ such that $\emptyset\subsetneq M'\subsetneq N'$, so that $N'$ cannot be an atom in $\C$.
Then $\mathfrak{Nr}_n\C$ ($2<n<\omega$) is atomic, but has no complete representation. 
By observing from the proof of the previous 
Theorem that $\Nr_n\CA_{\omega}\subseteq {\sf LCA}_n(={\bf El}\CRCA_n$) and similarly for $\sf RA$s, we have $\Ra\CA_{\omega}\subseteq {\sf LRRA}=({\bf El} \sf CRRA$), 
we get: 
the classes $\CRCA_n$ and $\sf CRRA$ are not elementary.
\end{proof}

Consider the statement:
{\it There exists a countable, complete and  atomic $L_n$ first order theory $T$ in a signature $L$ such that  the type $\Gamma$ 
consisting of co-atoms in the  cylindric Tarski-Lindenbaum quotient algebra $\Fm_T$ is realizable in every {\it $m$--square} model, 
but $\Gamma$ cannot be isolated using $\leq l$ variables, where $n\leq l<m\leq  \omega$.} A co-atom of $\Fm_T$ is the negation of an atom in $\Fm_T$, 
that is to say, is an element of the form $\Psi/\equiv_T$, where $\Psi/\equiv_T=(\neg \phi/{\equiv_T})=\sim (\phi/{\equiv_T})$ and $\phi/{\equiv_T}$ is an atom in $\Fm_T$ (for $L$-fomulas, $\phi$ and $\psi$).
Here the quotient algebra $\Fm_T$is formed relative to the congruence relation of semantical equivalence modulo $T$.
An  $m$-square model of $T$ is an $m$-square representation of $\Fm_T$.
The last statement  denoted by ${\sf not VT}(l, m)$, short for Vaught's Theorem ({\it $\sf VT$) fails at (the parameters) $l$ and $m$.}
Let ${\sf VT}(l, m)$ stand for {\it {\sf VT} holds at $l$ and $m$}, so that by definition ${\sf not VT}(l, m)\iff \neg {\sf VT}(l, m)$. We also include $l=\omega$ in the equation by 
defining ${\sf VT}(\omega, \omega)$ as {\sf VT} holds for $L_{\omega, \omega}$: Atomic countable first order theories have atomic countable models. 
We conjecture that $\sf VT$ {\it fails everywhere} in the sense that for the permitted values $n\leq l, m\leq \omega$, namely, 
for $n\leq l<m\leq \omega$  and $l=m=\omega$, 
$\VT(l, m)\iff l=m=\omega.$  In this direction we have the following strong partial result that seems to confirm our conjecture.

\begin{theorem}\label{fl} For $2<n<\omega$ and $n\leq l<\omega$, ${\sf not VT}(n, n+3)$ and ${\sf not VT}(l, \omega)$ hold.
Furthermore, if for each $n<m<\omega$, there exists a finite relation algebra $\R_m$ having $m-1$ strong blur and no
$m$-dimensional relational basis, then
for $2<n\leq l<m\leq \omega$  and $l=m=\omega$,  $\VT(l, m)\iff l=m=\omega$. 
\end{theorem}
\begin{proof}
We start by the last part. Let $\R_m$ be as in the hypothesis with strong $m-1$--blur $(J, E)$ and $m$-dimensional relational basis. 
We `blow up and blur' $\R_m$ in place of the Maddux algebra 
$\mathfrak{E}_k(2, 3)$ blown up and blurred in \cite[Lemma 5.1]{ANT}, where $k<\omega$ is the number of non--identity atoms 
and  $k$ depends recursively on $l$, giving the desired  $l$--blurness, cf. \cite[Lemmata  4.2, 4.3]{ANT}.
Now take $\A=\mathfrak{Bb}_n(\R_m, J, E)$ as defined in \cite{ANT} to be the $\CA_n$ obtained after blowing up and blurring $\R$ to a weakly representable atom structure $\cal R$. 
Here by \cite[Theorem 3.2 9(iii)]{ANT},  ${\sf Mat}_n\At\cal R$ (the set of $n$-basic matrices on $\At\cal R$) is a $\CA_n$ atom structure and 
$\A$ is an atomic  subalgebra of $\Cm{\sf Mat}_n(\At\cal R)$ containing $\Tm{\sf Mat}_n(\At\cal R)$, cf. \cite{ANT}.  
Then  $\A\in \RCA_n\cap \Nr_n\CA_l$ but $\A$ has no complete $m$-square representation.  
In fact,  by \cite[item (3) pp.80]{ANT},  $\A\cong \Nr_m{\mathfrak Bb}_l(\R_m, J, E)$.The last algebra ${\mathfrak Bb}_l(\R_m, J, E)$ is defined and the iomorphism holds because $\R_m$ has 
a  strong $l$-blur.
A complete $m$--square  representation of  
an atomic $\B\in \CA_n$ induces an $m$--square representation of $\Cm\At\B$.
%which implies by Theorem \ref{flat} that $\Cm\At\B\in \bold S\Nr_n{\sf D}_m$ which we know is not the case.
To see why, assume that $\B$ has an $m$--square  complete representation via $f:\B\to \D$, where $\D=\wp(V)$ and 
the base of the representation $\Mo=\bigcup_{s\in V} \rng(s)$ is $m$--square. Let $\C=\Cm\At\B$.
For $c\in C$, let $c\downarrow=\{a\in \At\C: a\leq c\}=\{a\in \At\B: a\leq c\}$. 
Define, representing $\C$,  
$g:\C\to \D$ by $g(c)=\sum_{x\in c\downarrow}f(x)$, then $g$ is the required homomorphism into $\wp(V)$ having base $\Mo$.
But $\Cm\At\A$ does not have an $m$-square representation, because $\R$ does not have an $m$-dimensional relational basis,  
and $\R\subseteq \Ra\Cm\At\A$. So an $m$-square representation of $\Cm\At\A$ induces one of $\R$ which t
hat $\R$ has no $m$-dimensional relational basis, a contradiction.

We prove ${\sf not Vt}(m-1, m)$, hence the required.
By \cite[\S 4.3]{HMT2}, we can (and will) assume that $\A= \Fm_T$ for a countable, simple and atomic theory $L_n$ theory $T$.  
Let $\Gamma$ be the $n$--type consisting of co--atoms of $T$. Then $\Gamma$ is realizable in every $m$--square model, for if $\Mo$ is an $m$--square model omitting 
$\Gamma$, then $\Mo$ would be the base of a complete  $m$--square  representation of $\A$, and so by Theorem \ref{flat} $\A\in \bold S_c\Nr_n{\sf D}_m$ which is impossible.
%But $\A\in {\sf Nr}_n\CA_l$, so  using the same  (terminology and) argument in \cite[Theorem 3.1]{ANT} we get that  
%any witness isolating $\Gamma$  needs more 
%than $l$--variables. 
Suppose for contradiction that $\phi$ is an $m-1$ witness, so that $T\models \phi\to \alpha$, for
all $\alpha\in \Gamma$, where recall that $\Gamma$ is the set of coatoms.
Then since $\A$ is simple, we can assume
without loss  that $\A$ is a set algebra with 
base $M$ say.
Let $\Mo=(M,R_i)_{i\in \omega}$  be the corresponding model (in a relational signature)
to this set algebra in the sense of \cite[\S 4.3]{HMT2}. Let $\phi^{\Mo}$ denote the set of all assignments satisfying $\phi$ in $\Mo$.
We have  $\Mo\models T$ and $\phi^{\Mo}\in \A$, because $\A\in \Nr_n\CA_{m-1}$.
But $T\models \exists x\phi$, hence $\phi^{\Mo}\neq 0,$
from which it follows that  $\phi^{\Mo}$ must intersect an atom $\alpha\in \A$ (recall that the latter is atomic).
Let $\psi$ be the formula, such that $\psi^{\Mo}=\alpha$. Then it cannot
be the case
that $T\models \phi\to \neg \psi$,
hence $\phi$ is not a  witness,
contradiction and we are done.
Finally, ${\sf not VT}(n, n+3)$ and ${\sf not VT}(l,  \omega)$ $(n\leq l<\omega)$ follow from Theorems \ref{can} and \ref{ANT}.
%like the proof of $\Psi(m-1,m)$ from the existence of an atomic countable simple 
%$\A\in \RCA_n\cap \Nr_n\CA_{m-1}$ whose \de\ completion having no complete $m$-square representation by replacing $m-1$ and $m$ by $n$ and $ n+3$ and $l$ and $\omega$, respectively. 
 \end{proof}

\end{document}